\documentclass[]{amsart}

\usepackage{graphicx,color,hyperref}
\numberwithin{equation}{section}

\usepackage[initials,nobysame]{amsrefs}

\newtheorem{theorem}{Theorem}[section]
\newtheorem{proposition}[theorem]{Proposition}
\newtheorem{corollary}[theorem]{Corollary}
\newtheorem{lemma}[theorem]{Lemma}
\theoremstyle{definition}
\newtheorem{definition}[theorem]{Definition}

\newtheorem{problem}[theorem]{Problem}
\theoremstyle{remark}
\newtheorem{remark}[theorem]{Remark}

\DeclareMathOperator{\sech}{sech}
\DeclareMathOperator{\cn}{cn}

\DeclareMathOperator{\sn}{sn}
\DeclareMathOperator{\am}{am}

\usepackage{braket} 
\usepackage{amssymb} 
\usepackage{color}
\usepackage{mathrsfs} 
\usepackage{url}

\newcommand{\R}{\mathbf{R}}
\newcommand{\Z}{\mathbf{Z}}

\newcommand{\N}{\mathbf{N}}

\newcommand{\vp}{\varphi}

\newcommand{\B}{\mathcal{B}}
\newcommand{\K}{\mathrm{K}}
\newcommand{\E}{\mathrm{E}}
\newcommand{\Ap}{\mathcal{A}_\mathrm{pin}}
\newcommand{\Ah}{\mathcal{A}_\mathrm{hook}}

\newcommand{\amcn}{\am_{1,p}}

\newcommand{\Ecn}{\mathrm{E}_{1,p}}
\newcommand{\Fcn}{\mathrm{F}_{1,p}}
\newcommand{\Kcn}{\mathrm{K}_{1,p}}

\begin{document}

\title{Variational stabilization of degenerate $p$-elasticae}

\author{Tatsuya Miura}
\address[T.~Miura]{Department of Mathematics, Tokyo Institute of Technology, 2-12-1 Ookayama, Meguro-ku, Tokyo 152-8551, Japan \\
Current address: Department of Mathematics, Graduate School of Science, Kyoto University,
Kitashirakawa Oikawa-cho, Sakyo-ku, Kyoto 606-8502, Japan}
\email{tatsuya.miura@math.kyoto-u.ac.jp}

\author{Kensuke Yoshizawa}
\address[K.~Yoshizawa]{Faculty of Education, Nagasaki University, 1-14 Bunkyo-machi, Nagasaki, 852-8521, Japan}
\email{k-yoshizaw@nagasaki-u.ac.jp
}


\date{\today}

\begin{abstract}
    A new stabilization phenomenon induced by degenerate diffusion is discovered in the context of pinned planar $p$-elasticae.
    It was known that in the non-degenerate regime $p\in(1,2]$, including the classical case of Euler's elastica, there are no local minimizers other than unique global minimizers.
    Here we prove that, in stark contrast, in the degenerate regime $p\in(2,\infty)$ there emerge uncountably many local minimizers with diverging energy.
\end{abstract}


\maketitle

\section{Introduction}

Nonlinear diffusion often provokes peculiar behavior of solutions, leading to new analytical challenges.
In this paper we discover a strong variational stabilization phenomenon due to (1D) degenerate diffusion in the context of pinned planar $p$-elasticae.
More precisely, we study the structure of local minimizers of the $p$-bending energy under the fixed-length constraint and the pinned boundary condition.
Let $p\in(1,\infty)$.
The $p$-bending energy is a fundamental geometric energy defined by
$$\B_p[\gamma]:=\int_\gamma |k|^p\,ds,$$
where $\gamma$ is a planar curve, $k$ is the signed curvature, and $s$ is the arclength parameter.
For $L>0$ we define
$$W^{2,p}_\mathrm{arc}(0,L;\R^2):=\Set{ \gamma \in W^{2,p}(0,L; \R^2) | {} |\gamma'|\equiv1 },$$
the set of arclength parametrized curves of length $L$ in the $W^{2,p}$-Sobolev class.
Given $P_0,P_1\in\R^2$ such that $|P_1-P_0|<L$, we define the admissible space $\Ap$ by
\begin{equation}\notag
    \Ap=\Ap(P_0,P_1,L):=
    \Set{\gamma \in W^{2,p}_\mathrm{arc}(0,L; \R^2) | \gamma(0)=P_0, \ \gamma(L)=P_1 },
\end{equation}
equipped with the $W^{2,p}$-Sobolev metric.
(Recall that $W^{2,p}(0,L)\subset C^1([0,L])$.)
As usual, $\gamma\in\Ap$ is called a global minimizer (resp.\ local minimizer) if $\B_p[\eta]\geq\B_p[\gamma]$ holds for all $\eta\in\Ap$ (resp.\ for all $\eta\in\Ap$ in a neighborhood of $\gamma$).
Previous results ensure that there are infinitely many critical points, and among them global minimizers are unique up to isometries \cite{MR4852697}.
As for local minimality (stability), the pinned boundary condition allows any rotation of the tangent directions at the endpoints, so that critical points have a strong predisposition towards instability.
Indeed, in the classical case $p=2$, it was known that all but global minimizers are unstable; see e.g.\ Maddocks' linear stability analysis in 1984 \cite{Maddocks1984}.
For $p\neq2$, stability analysis is much more delicate since in general the $p$-bending energy involves a generic loss of regularity (where $k=0$) and is not compatible with standard second-variation arguments due to the lack of Hilbert structure.
However, we recently succeeded in extending the classical instability result to the case $p\in(1,2]$ as well as a part of the case $p\in(2,\infty)$ by a totally new approach.

\begin{theorem}[\cite{MY_Crelle}]\label{thm:previous_rigidity}
    If $p\in(1,2]$, or if $p\in(2,\infty)$ and $|P_1-P_0|\leq \frac{1}{p-1}L$, then there are no local minimizers of $\B_p$ in $\Ap$ other than global minimizers.
\end{theorem}

Up to now it was completely open whether there exists a nontrivial local minimizer in the remaining case.
In this paper we discover that the remaining case is surprisingly different from the classical case, proving that not only there emerge nontrivial local minimizers but also their energy can be arbitrarily large.

\begin{theorem}\label{thm:main_stabilization}
    If $p\in(2,\infty)$ and $|P_1-P_0|>\frac{1}{p-1}L$, then there exists an uncountable family $\{\gamma_\delta\}_{\delta}$ of local minimizers of $\B_p$ in $\Ap$ such that $\sup_{\delta}\B_p[\gamma_\delta]=\infty$.
\end{theorem}

This phenomenon is strongly due to degeneracy in the nonlinear diffusion term of the Euler--Lagrange equation.
Our result shows for the first time that degeneracy induces new local minimizers in $p$-elastica theory.
Even in a broader context of nonlinear diffusion, to the authors' knowledge, our theorem seems to provide the first (or at least unusual) example that degeneracy yields new local minimizers of arbitrarily high energy.
Below we discuss this point in more details.

Nonlinear diffusion equations are differential equations with nonlinear generalizations of the (linear) Laplacian $\Delta$ in the top-order term.
The $p$-Laplacian 
$\Delta_p u := \nabla\cdot(|\nabla u|^{p-2}\nabla u)$
is a typical example of a nonlinear diffusion operator, arising from the first variation of the $p$-Dirichlet energy
$\int_\Omega |\nabla u|^p dx.$
(See e.g.\ the books \cites{DiB93_book, HKM06_book, Vaz06_book, Lind19} for details.)
The case of $p>2$ is said to be \emph{degenerate} since the ellipticity-related gradient term $|\nabla u|^{p-2}$ in $\Delta_p$ may vanish.
In the degenerate case, solutions to elliptic equations involving $\Delta_p$ may have nontrivial regions where the gradient vanishes, called \emph{flat core} \cite{KV93, Takeuchi12}, 
while solutions to parabolic equations may exhibit \emph{slow diffusion} with finite-speed propagation of free boundaries on the flat part \cite{Bar52, KV88} (see also \cite{CV21}).
Similar phenomena are also observed for the porous medium equation $u_t=\Delta(|u|^{m-1}u)$ with $m>1$.
Roughly speaking, degenerate diffusion makes solutions prefer to be flat, which can cause new stabilization effects.
Degeneracy-induced transitions at $p=2$ also appear in various ways, see e.g.\ Figalli--Zhang's recent stability result on the Sobolev inequality \cite{FigalliZhang22}.

Our problem here is concerned with \emph{$p$-elasticae}, i.e., critical points of the $p$-bending energy $\B_p$ under the fixed-length constraint.
The classical case $p=2$ corresponds to the celebrated problem of Euler's elastica studied from the 18th century (see e.g.\ \cite{Tru83, Lev, Sin, Sac08_2, Miura20}).
The quadratic bending energy arises from standard linear elasticity, while non-quadratic bending energies also appear in several contexts such as DNA cyclization \cite{DSSVV05} and image processing \cite{MM98, AM03, DFLM}.
In fact, even for $p\neq2$, there are also many studies on $p$-elasticae (e.g.\ \cite{nabe14, LP22, MY_AMPA, MR4852697, MY_Crelle}) as well as related problems on the $p$-bending energy (e.g.\ \cite{BDP93, AGM, BM04, BM07, MN13, AM17, FKN18,  NP20, OPW20, Poz20, SW20, BHV, BVH, Poz22, GPT23, OW23, DMOY}).
The Euler--Lagrange equation for $p$-elastica is formally given by
\begin{equation}\label{eq:p-elastica}
    p(|k|^{p-2}k)_{ss} +(p-1)|k|^pk - \lambda k =0, \quad \lambda \in \R.
\end{equation}
If $p>2$, this equation also involves degeneracy where $k$ vanishes.
Here we point out that the top-order operator in \eqref{eq:p-elastica} is rather of porous medium type $(|u|^{p-2}u)''$, while in terms of tangential angle $\theta$ the Euler--Lagrange equation involves exactly the $p$-Laplacian $(|u'|^{p-2}u')'$ \cite{nabe14}:
$$(|\theta_s|^{p-2}\theta_s)_s=R\sin(\theta+\alpha).$$
The degeneracy as well as singularity leads to generic loss of regularity and hence the analytical treatment of $p$-elastica is much more delicate than the classical one.
Building on Watanabe's previous work \cite{nabe14}, the authors recently obtained a complete classification of planar $p$-elasticae with explicit formulae and optimal regularity \cite{MY_AMPA}.
This in particular shows that in the non-degenerate regime $p\in(1,2]$ the zero set of the curvature $k$ is always discrete, while in the degenerate regime $p\in(2,\infty)$ the zero set may contain nontrivial intervals, which are also called flat core in this context.
In our subsequent study on the pinned boundary value problem \cite{MR4852697}, we classified all pinned $p$-elasticae and proved that any global minimizer is given by a convex arc, unique up to invariances.

Stability of $p$-elasticae is much more delicate.
In \cite{MY_Crelle} the authors developed a new geometric method to prove general instability results, which in particular imply Theorem \ref{thm:previous_rigidity}.
In the remaining case of Theorem \ref{thm:previous_rigidity}, our instability theory still partially works, implying that any local minimizer is necessarily either a global minimizer (convex arc) or a \emph{quasi-alternating} flat-core pinned $p$-elastica \cite{MY_Crelle}*{Corollary 2.14}; in particular, this implies the instability of (i) and (ii) in Figure~\ref{fig:quasi-flat} (see also Remark~\ref{rem:threshold-stability}).
However, it was totally open whether there indeed exists a stable quasi-alternating elastica.
Our main result, from which Theorem \ref{thm:main_stabilization} directly follows, resolves this problem in a generic sense:

\begin{theorem}[Theorem \ref{thm:alternating_detailed}]\label{thm:alternating_stability}
    Every alternating flat-core pinned $p$-elastica is stable.
\end{theorem}

Roughly speaking, an \emph{alternating} flat-core $p$-elastica has segments and loops located alternatingly, with no loop touching an endpoint (see Definition \ref{def:alternating} and Figure \ref{fig:quasi-flat} (iii)).
The quasi-alternating class additionally allows segments to degenerate to points whenever the two neighboring loops are in a same direction (Figure \ref{fig:quasi-flat} (iv)).
Therefore, the alternating class can naturally be regarded as a generic subclass of the quasi-alternating class (see also Remark \ref{rem:simplex}).
\begin{figure}[htbp]
\centering
\includegraphics[width=120mm]{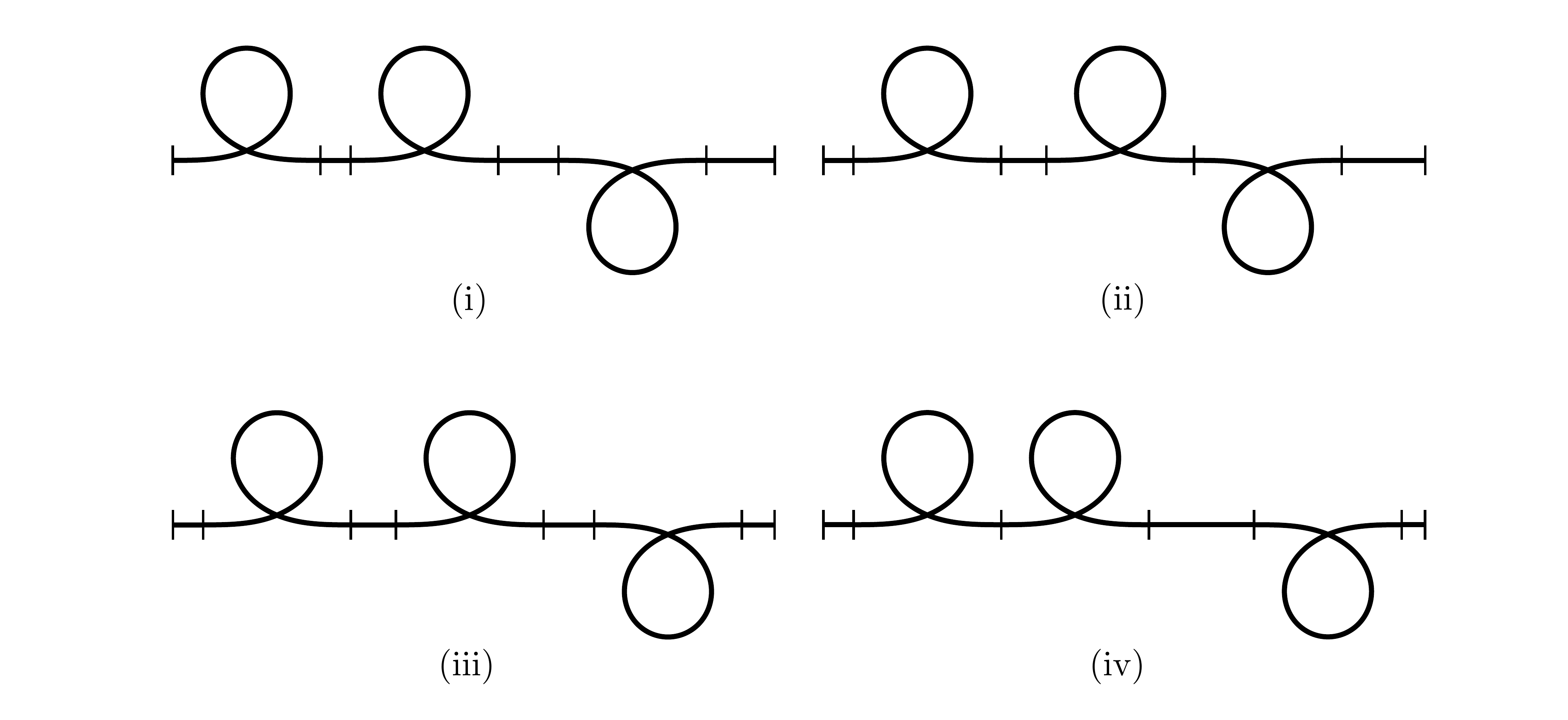} 
\caption{Examples of flat-core pinned $p$-elasticae.
(i), (ii) Unstable \cite{MY_Crelle}.
(iii) Stable (alternating, Theorem~\ref{thm:alternating_stability}).
(iv) Open (quasi-alternating, Problem~\ref{prob:quasi-open}).
}
\label{fig:quasi-flat}
\end{figure}

Among many others, we briefly mention some closely related results for reaction-diffusion equations involving the $p$-Laplacian and a logistic-type nonlinearity, originating from e.g.\ \cite{GV88, KV93} (see also \cite{HHH20} and references therein).
It is well known that those equations may have stationary solutions with flat parts if $p>2$.
Their (dynamical) stability is however not yet completely understood even in 1D, in particular in the sign-changing case (with many flat parts).
Since such solutions can be parametrized by the position of the transition layers, one can regard the parameter spaces as simplices \cite{GV88}*{Theorem 2.2}.
In terms of this simplex, Takeuchi--Yamada \cite{TY00} 
proved instability in a boundary case, and later Takeuchi \cite{Tak00} 
proved a partial stability in an interior case, restricting the neighborhood.
Full stability in the interior case is a long-standing open problem.
Our results may be regarded as certain variational counterparts, since the tangential angle of a flat-core (or borderline) elastica may be regarded as a transition layer (cf.\ \cite{Miura20} for $p=2$). In particular, our present result gives a full variational stability in the interior case.

We now explain the key idea for the proof of Theorem \ref{thm:alternating_stability}.
As mentioned, stability of $p$-elasticae cannot be tackled by using standard theory in contrast to the case of $p=2$, cf.\ \cite{Maddocks1981,Maddocks1984,Sac08_2,Sac08_3}.
(See also \cite{Rynne20,Rynne19} for similar subtlety in dynamical stability.)
Our proof proceeds in a nonstandard way; the local-minimality issue is reduced to a new (global) minimization problem subject to an auxiliary free boundary problem through a geometric relaxation procedure.
More precisely, we first cut an alternating flat-core pinned $p$-elastica at the top of each loop and also an interior point of each inner segment.
Then we prove that each piece of curve has an independent minimizing property even after a perturbation.
This implies the desired local minimality of the whole curve.
In this procedure, a straightforward choice of the boundary condition for each piece would be a kind of clamped boundary conditions (prescribing up to first order) in order to retain the admissibility of the whole curve.
However, for such a choice it is usually hard to detect minimizers and compute their energy, cf.\ \cite{Miura20}, even though there are explicit formulae for critical points.
Our main ingenuity here lies in the delicate choice of the relaxed boundary condition, which we call the \emph{hooked boundary condition}.
This choice turns out to be very well suited:
On one hand, it is so relaxed that we can use additional natural boundary conditions for restricting all critical points in ``computable'' forms.
Here we heavily rely on our previous classification theory \cite{MY_AMPA, MR4852697}, which describes $p$-elasticae in terms of $p$-elliptic functions introduced by the authors and of $p$-elliptic integrals introduced by Watanabe \cite{nabe14} with reference to Takeuchi's work \cite{Takeuchi_RIMS}.
Armed with those tools and resulting monotonicity properties, we can explicitly obtain unique global minimizers and represent their energy by complete $p$-elliptic integrals (Theorem \ref{thm:min_free_half}).
On the other hand, it is \emph{not} so relaxed that the minimality of the original curve cannot be recovered, even though the above relaxation allows our admissible set to include even discontinuous competitors (Lemma \ref{lem:relaxation}).
The last minimality is eventually reduced to an elementary argument involving Jensen's inequality, which however crucially reflects the effect of degeneracy.

It would be worth mentioning that the core of our relaxation technique is in the same spirit of some recent studies in different contexts, although the details are quite independent; an isoperimetric inequality for multiply-winding curves by \cite{MO21}, and on a Li--Yau type inequality in terms of the bending energy \cite{Miura_LiYau} as well as the $p$-bending energy \cite{MR4852697}.
The main characteristic here is that local minimality is reduced to global minimality.
This point is subtle, and accordingly our study does not cover the full quasi-alternating class, leaving the following problem open: 

\begin{problem}\label{prob:quasi-open}
    Is every quasi-alternating flat-core pinned $p$-elastica stable?
\end{problem}

This paper is organized as follows:
In Section~\ref{sect:preliminary} we prepare notation for $p$-elliptic functions and recall some known results for $p$-elasticae.
In Section~\ref{sect:hooked} we introduce the hooked boundary condition and classify all hooked $p$-elasticae (Theorem~\ref{thm:classify-hook}).
Uniqueness of minimal hooked $p$-elasticae will be also deduced.
In Section~\ref{sect:proof-MainResults} we complete the proof of Theorems~\ref{thm:main_stabilization} and \ref{thm:alternating_stability}.

\subsection*{Acknowledgements}
The first author is supported by JSPS KAKENHI Grant Numbers 20K14341, 21H00990, and 23H00085 and by Grant for Basic Science Research Projects from The Sumitomo Foundation.
The second author is supported by JSPS KAKENHI Grant Number 22K20339.

\section{Preliminary} \label{sect:preliminary}

In this section we first recall from \cite{MY_AMPA} the definitions and fundamental properties of $p$-elliptic integrals and functions, and also some known facts for planar $p$-elasticae.

Throughout this paper, let $e_1,e_2\in\R^2$ be the canonical basis, and $p\in(1,\infty)$ unless specified otherwise.
For an arclength parametrized planar curve $\gamma:[0,L]\to\R^2$, the function $\theta:[0,L]\to\R$ denotes the tangential angle $\partial_s\gamma=(\cos\theta,\sin\theta)$, and $k:[0,L]\to\R$ denotes the (counterclockwise) signed curvature $k=\partial_s\theta$.
In addition, $R_\phi$ denotes the counterclockwise rotation matrix through angle $\phi\in \R$.

\subsection{$p$-Elliptic integrals and functions}

Now we recall the definitions of $p$-elliptic integrals introduced in \cite{nabe14}.
Note that there are two types of generalizations, for example $\E_{1,p}$ and $\E_{2,p}$, but here we only use the first type such as $\E_{1,p}$.

\begin{definition} \label{def:K_p-E_p}
The \emph{incomplete $p$-elliptic integral of the first kind} $\mathrm{F}_{1,p}(x,q)$ of modulus $q\in[0,1)$ is defined for $x\in\R$ by 
\begin{align*}
&\mathrm{F}_{1,p}(x, q):=\int_0^{x} \frac{|\cos\phi|^{1-\frac{2}{p}} }{ \sqrt[]{1-q^2\sin^2\phi} }\,d\phi, 
\end{align*}
and also the corresponding \emph{complete $p$-elliptic integral} $\K_{1,p}(q)$ by 
\begin{align*}
    \K_{1,p}(q):=\mathrm{F}_{1,p}(\pi/2, q).
\end{align*}
For $q=1$, they are defined by
\[
\mathrm{F}_{1,p}(x,1):=\displaystyle \int_0^{x} \frac{d\phi}{|\cos \phi|^{\frac{2}{p} } }, \quad \text{where}\ 
\begin{cases}
x\in(-\frac{\pi}{2},\frac{\pi}{2}) \quad &\text{if} \ \ 1< p \leq  2,  \\
x\in\R &\text{if} \ \ p>2,
\end{cases}
\]
and
\begin{align*} 
\K_{1,p}(1)=\K_{p} (1):= 
\begin{cases}
\infty \quad &\text{if} \ \ 1< p \leq  2,  \\
\displaystyle \int_0^{\frac{\pi}{2}} \frac{d\phi}{(\cos \phi)^{\frac{2}{p} } } < \infty &\text{if} \ \ p>2.
\end{cases}
\end{align*}
Also, the \emph{incomplete $p$-elliptic integral of the second kind} $\E_{1,p}(x,q)$ of modulus $q\in[0,1]$ is defined for $x\in\R$ by 
\[
\E_{1,p}(x,q):=\int_0^{x} \sqrt{1-q^2 \sin^2 \phi}\, |\cos \phi|^{1-\frac{2}{p}}\,d\phi, 
\]
and also the corresponding \emph{complete $p$-elliptic integral} $\E_{1,p}(q)$  by 
\begin{align*}
    \E_{1,p}(q):=\E_{1,p}(\pi/2, q).
\end{align*}
\end{definition}

Note that, by definition and periodicity, for any $x\in \R$, $q\in[0,1)$, and $n\in \Z$, 
\begin{align}\label{eq:period_Ecn}
\begin{split}
\Ecn(x+ n\pi , q) &=  \Ecn(x , q) + 2n\Ecn(q), \\
\Fcn(x+ n\pi , q) &=  \Fcn(x , q) + 2n\Kcn(q).
\end{split}
\end{align}

The following lemma will play fundamental roles.

\begin{lemma}[{\cite{nabe14}*{Lemma 2}}]\label{lem:nabe-lem2}
Let $Q_p:[0,1)\to \R$ be defined by
\begin{align} \notag
Q_p(q):= 2 \frac{\E_{1,p}(q)}{\K_{1,p}(q)}-1 , \quad q\in[0,1).
\end{align}
Then $Q_p$ is strictly decreasing on $[0,1)$. 
Moreover, $Q_p$ satisfies $Q_p(0)=1$ and
\begin{align} \notag
\lim_{q\uparrow1}Q_p(q)=
\begin{cases}
-1 \quad & \text{if} \ \ 1<p \leq 2, \\
-\dfrac{1}{p-1} & \text{if} \ \ p>2.
\end{cases}
\end{align}
\end{lemma}

Next we recall $p$-elliptic and $p$-hyperbolic functions introduced in \cite{MY_AMPA}.
Here we focus on those we will use later; for example, we define $\sech_p$ only for $p>2$ (see \cite{MY_AMPA}*{Definition 3.8} for $p\in(1,2]$).

\begin{definition} \label{def:cndn}
Let $ q \in[0, 1]$.
The \emph{amplitude function} $\amcn(x,q)$ with modulus $q$ is defined by the inverse functions of $\mathrm{F}_{1,p}(x, q)$, i.e., for $x\in\R$, 
\begin{align} \nonumber
x= \int_0^{\amcn(x,q)} \frac{|\cos\phi|^{1-\frac{2}{p}} }{ \sqrt[]{1-q^2\sin^2\phi} }\,d\phi.
\end{align}
The \emph{$p$-elliptic sine} $\sn_p(x,q)$ with modulus $q$ is defined by 
\begin{align}\notag
\sn_p (x, q):= \sin \amcn(x,q), \quad x\in \R.
\end{align}
The \emph{$p$-elliptic cosine} $\cn_p(x,q)$ with modulus $q$ is defined by 
\begin{align}\label{eq:cn_p}
\cn_p (x, q):= |\cos \amcn(x,q)|^{\frac{2}{p}-1} \cos \amcn(x,q), \quad x\in \R.
\end{align}
%
For $p>2$, the \emph{$p$-hyperbolic secant} $\sech_p x$ is defined by
\begin{align}\label{eq:sech_p} 
 \sech_p x:= 
\begin{cases}
 \cn_p(x,1), \quad & x\in (-\K_p(1), \K_p(1)), \\
 0, &x\in \R \setminus (-\K_p(1), \K_p(1)).
\end{cases}
\end{align}
Moreover, the \emph{$p$-hyperbolic tangent} $\tanh_p x$ is defined by
\[\tanh_p x:=\int_0^x (\sech_p t)^p dt, \quad x\in \R. \]
\end{definition}

\begin{proposition}[{\cite[Proposition 3.10]{MY_AMPA}}]\label{prop:property-cndn}
Let $\cn_p$ and $\sech_p$ be given by \eqref{eq:cn_p} and \eqref{eq:sech_p}, respectively. 
Then the following statements hold:
\begin{itemize}
\item [(i)] For $q\in[0,1)$, $\cn_p(\cdot, q)$ is an even $2\K_{1,p}(q)$-antiperiodic function on $\R$ and, in $[0, 2\K_{1,p}(q)]$, strictly decreasing from $1$ to $-1$. 
\item [(ii)] 
Let $p> 2$. 
Then $\sech_p $ is an even nonnegative function on $\R$, and strictly decreasing in $[0, \K_{p} (1))$. 
Moreover, $\sech_p 0=1$ and 
$\sech_p x \to 0$
as $x \uparrow \K_{p} (1)$.
In particular, $\sech_p $ is continuous on $\R$.
\end{itemize}
\end{proposition}

In particular, this with $\cn_p(\Kcn(q),q)=0$ implies that
\begin{align}\label{eq:Z_pq}
Z_{p,q}:=\Set{x\in \R | \cn_p(x,q)=0} = 
(2\Z+1)\Kcn(q).
\end{align}

\subsection{$p$-Elasticae}
In this paper we call an arclength parametrized planar curve \emph{$p$-elastica} if its signed curvature $k:[0,L]\to\R$ (parametrized by the arclength) belongs to $L^\infty(0,L)$ and if there is $\lambda\in\R$ such that $k$ satisfies
\begin{align}\label{eq:EL}\tag{EL}
\int_0^L \Big( p |k|^{p-2}k \vp'' +(p-1) |k|^pk\vp -\lambda k \vp \Big) ds = 0
\end{align}
for any $\varphi \in C^\infty_{\rm c}(0,L)$.
Equation \eqref{eq:EL} is a weak form of \eqref{eq:p-elastica} and appears as the Euler--Lagrange equation for critical points of $\mathcal{B}_p$ under the fixed-length constraint.

As opposed to $p=2$, the $p$-elasticae are not necessarily smooth.
However at least the curvature is continuous, and also the curvature to a suitable power has more regularity enough to solve an 
Euler--Lagrange equation in the classical sense.

\begin{proposition}[{\cite{MY_AMPA}*{Theorem 1.7 and Lemma~4.3}}]\label{prop:regularity-w}
    Let $p\in(1,\infty)$.
    If an arclength parametrized curve $\gamma:[0,L]\to\R^2$ is a $p$-elastica, then $\gamma$ has the signed curvature $k$ such that $k\in C([0,L])$. 
    In addition, $w:=|k|^{p-2}k\in C^2([0,L])$ and
    \begin{align}\label{eq:0610-1}
    p w'' + (p-1)|w|^{\frac{2}{p-1}}w - \lambda |w|^{\frac{2-p}{p-1}}w =0 \quad \text{in} \ [0,L].
    \end{align}
\end{proposition}

Throughout this paper we also crucially use our explicit formulae for $p$-elasticae, focusing on those with vanishing curvature.
To describe them, we introduce the following notation on a concatenation of curves.
For $\gamma_j:[a_j,b_j]\to\R^2$ with $L_j:=b_j-a_j\geq0$, we define $\gamma_1\oplus\gamma_2:[0,L_1+L_2]\to\R^2$ by
\begin{align*}
  (\gamma_1\oplus\gamma_2)(s) :=
  \begin{cases}
    \gamma_1(s+a_1), \quad & s \in[0, L_1], \\
    \gamma_2(s+a_2-L_1) +\gamma_1(b_1)-\gamma_2(a_2),  & s \in[L_1,L_1+L_2].
  \end{cases}
\end{align*} 
We inductively define $\gamma_1\oplus\dots\oplus\gamma_N := (\gamma_1\oplus\dots\oplus\gamma_{N-1})\oplus\gamma_N$.
We also write
\begin{align}\notag
     \bigoplus_{j=1}^N\gamma_j:=\gamma_1\oplus\dots\oplus\gamma_N.
\end{align}

\begin{proposition}[\cite{MY_AMPA}*{Theorems~1.2, 1.3}]\label{prop:MY2203-thm1.1}
  Let $L>0$ and $\gamma\in W^{2,p}_{\rm arc}(0,L;\R^2)$ be a $p$-elastica whose signed curvature $k$ has a zero in $[0,L]$. 
  Then, up to similarity (i.e., translation, rotation, reflection, and dilation) and reparametrization, the curve $\gamma$ is represented by $\gamma(s)=\gamma_*(s+s_0)$ with some $s_0\in\R$, where $\gamma_*:\R\to\R^2$ is one of the arclength parametrized curves $\gamma_\ell$, $\gamma_w$, $\gamma_f$ defined as follows:
   \begin{itemize}
     \item \textup{(Linear $p$-elastica)}
     $\gamma_\ell(s):=(s,0)$.
     \item \textup{(Wavelike $p$-elastica)}
     For some $q \in (0,1)$,
     \begin{equation}\label{eq:EP2}
       \gamma_w(s):=\gamma_w(s,q) =
       \begin{pmatrix}
       2 \E_{1,p}(\am_{1,p}(s,q),q )-  s  \\
       -q\frac{p}{p-1}|\cn_p(s,q)|^{p-2}\cn_p(s,q)
       \end{pmatrix}.
     \end{equation}
     In this case, the tangential angle is given by $\theta_w(s)=2\arcsin(q\sn_p(s,q))$ and the signed curvature by $k_w(s) = 2q\cn_p(s,q).$
       \item \textup{(Flat-core $p$-elastica: $p>2$)}
    For some integer $N\geq1$, signs $\sigma_1,\dots,\sigma_N\in\{+,-\}$, and nonnegative numbers $L_1,\dots,L_N\geq0$,
    \begin{equation}\notag
      \gamma_f := \bigoplus_{j=1}^N (\gamma_\ell^{L_j}\oplus \gamma_b^{\sigma_j}),
    \end{equation}
    where $\gamma_b^\pm:[-\K_p(1),\K_p(1)]\to\R^2$ and $\gamma_\ell^{L_j}:[0,L_j]\to\R^2$ are defined by
    \begin{align}\notag
      \gamma_b^{\pm}(s) =
      \begin{pmatrix}
      2 \tanh_p{s} - s  \\
      \mp \frac{p}{p-1}(\sech_p{s})^{p-1}
      \end{pmatrix},
      \quad
      \gamma_\ell^{L_j}(s) =
      \begin{pmatrix}
      -s  \\
      0
      \end{pmatrix}.
    \end{align}
    The curves $\gamma_b^{\pm}(s)$ have the tangential angles $\theta_b^\pm(s)=\pm2\am_{1,p}(s,1)$ and the signed curvatures $k_b^\pm(s) = \pm2\sech_p{s}$ for $s\in [-\K_p(1),\K_p(1)]$.
    In particular, the signed curvature of $\gamma_f$ is given by
         \begin{align}\label{eq:flat-type}
                k_f(s)=\sum_{j=1}^N\sigma_j 2\sech_p(s-s_j), \quad \text{where}\quad  s_j=(2j-1)\K_p(1) + \sum_{i=1}^j L_i.     
         \end{align}
   \end{itemize}
\end{proposition}

\begin{remark}
    The vanishing curvature condition rules out orbitlike and borderline $p$-elasticae obtained in \cite{MY_AMPA}*{Theorems~1.2, 1.3}.
\end{remark}

\begin{remark}\label{rem:loop}
    The curve $\gamma^+_b$ looks like a (finite) loop as in Figure \ref{fig:gamma_b_v2}.
    In particular, the tangential angle $\theta_b^+:[-\K_p(1),\K_p(1)]$ is strictly monotone from $-\pi$ to $\pi$.
    By \cite{MY_AMPA}*{Lemma~5.7} and symmetry we also deduce that $\gamma_b^+(\pm\K_p(1))=\mp\frac{\K_p(1)}{p-1}e_1$.
\end{remark}

\begin{figure}[htbp]
\centering
\includegraphics[width=60mm]{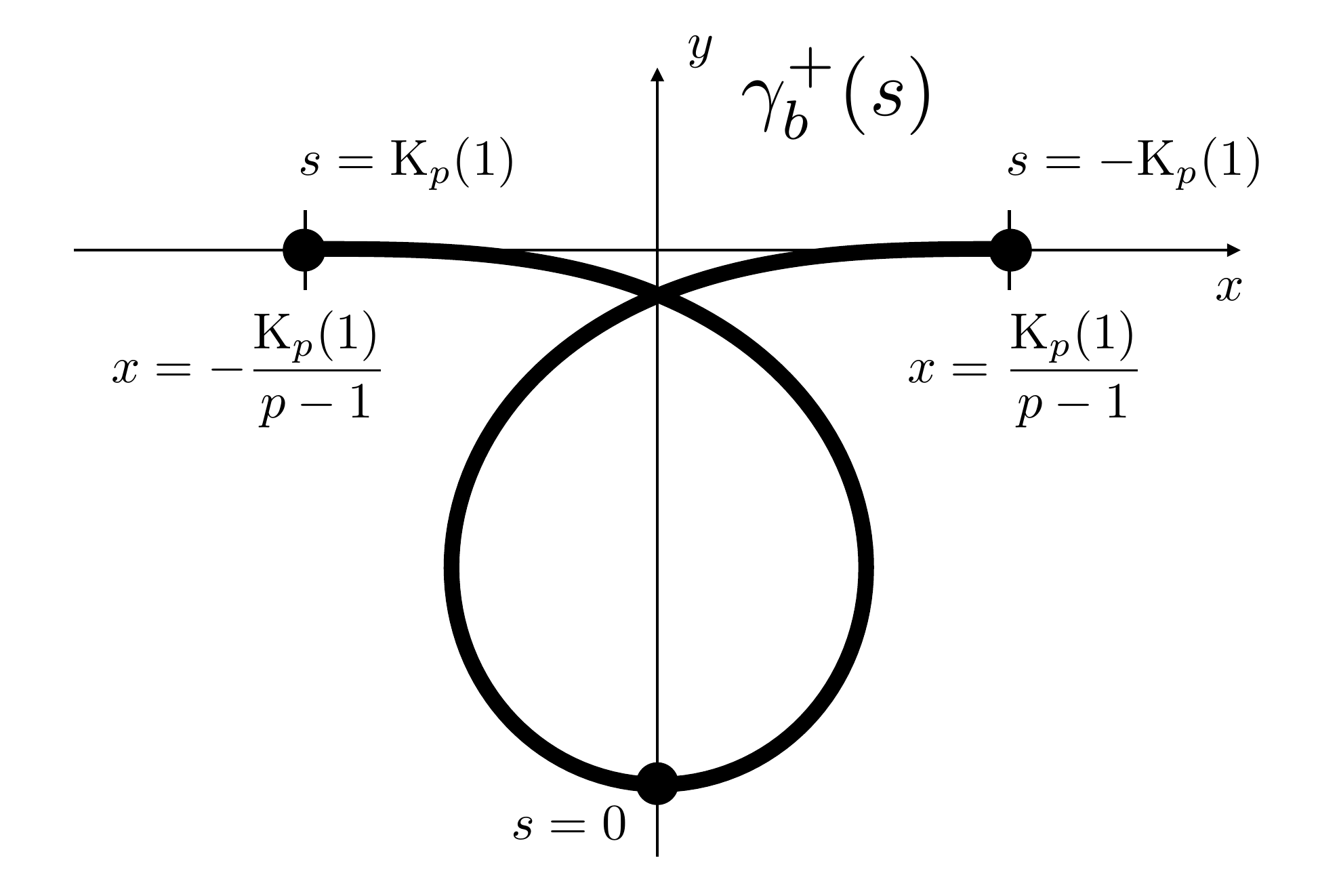} 
\caption{The profile of the loop $\gamma^+_b$.}
\label{fig:gamma_b_v2}
\end{figure}

In addition,  we call $\gamma\in\Ap$ a \emph{pinned $p$-elastica} if $\gamma$ is a $p$-elastica with $k(0)=k(L)=0$.
This is the standard first-order necessary condition to be a local minimizer of $\B_p$ in $\Ap$ (see \cite{MR4852697} for details).

By \cite{MR4852697}*{Theorem~1.1} we know that there are flat-core pinned $p$-elasticae.
To be more precise, suppose that $p>2$ and $|P_1-P_0|\in[\frac{1}{p-1}L,L)$. 
Then $\gamma\in\Ap$ is a pinned $p$-elastica if there are $N\in \N$, $r\in[\frac{1}{p-1},1)$, $\boldsymbol{\sigma}=(\sigma_1, \ldots, \sigma_N) \in \{+,-\}^N$, and $\boldsymbol{L}=(L_1, \ldots, L_{N+1}) \in [0,\infty)^{N+1}$ such that, up to similarity and reparametrization, the curve $\gamma$ is given by
    \begin{align} \label{eq:N-loop-flat-core}
        \gamma_{\rm flat}:=\bigg( \bigoplus_{j=1}^N \big( \gamma_{\ell}^{L_j} \oplus \gamma_{b}^{\sigma_j} \big) \bigg) \oplus \gamma_{\ell}^{L_{N+1}}, 
    \end{align}
and in addition, the numbers $p, r, N$ and $\boldsymbol{L}$ satisfy  
    \begin{align} \label{eq:sum-flatparts}
    \sum_{j=1}^{N+1}L_j = 2N\frac{r-\frac{1}{p-1}}{1-r} \K_{p}(1).
    \end{align}
Notice that the length $\bar{L}$ of $\gamma_\mathrm{flat}$ is given by $\bar{L}=2N\K_p(1)+\sum_{j=1}^{N+1} L_j$, and
\begin{gather}\label{eq:flcr-property}
\begin{split}
    \gamma_{\rm flat} (\bar{L}) - \gamma_{\rm flat}(0) &= -\bigg(\frac{2N}{p-1}\K_{p}(1) + \sum_{j=1}^{N+1} L_j \bigg)e_1,\\
    \gamma_{\rm flat}'(\bar{L})=\gamma_{\rm flat}'(0)&=-e_1.
\end{split}
\end{gather}
In particular, since $\gamma\in\Ap$, we need to have $r=\frac{|P_0-P_1|}{L}$.
On the other hand, $\boldsymbol{\sigma}$ is arbitrary, and also $N$ and $\boldsymbol{L}$ are arbitrary whenever \eqref{eq:sum-flatparts} holds.



Finally we recall the definition of the class of alternating flat-core $p$-elasticae introduced in \cite{MY_Crelle}*{Section~6.3}.

\begin{definition}[Alternating flat-core]\label{def:alternating}
Let $p>2$, $\frac{1}{p-1}L<|P_1-P_0|<L$, and $N\in \N$.
We call $\gamma\in \Ap$ an \emph{$N$-loop alternating flat-core $p$-elastica} if, up to similarity and reparametrization, the curve $\gamma$ is of the form \eqref{eq:N-loop-flat-core} for $
r=\frac{|P_0-P_1|}{L}$, for some $\boldsymbol{\sigma}=(\sigma_1, \ldots, \sigma_N) \in \{+,-\}^N$, and for some strictly positive numbers $\boldsymbol{L}=(L_1, \ldots, L_{N+1}) \in (0,\infty)^{N+1}$ satisfying \eqref{eq:sum-flatparts}.
 \end{definition}

Thanks to the strict positivity of $\boldsymbol{L}$, any alternating flat-core $p$-elastica has the segments and the loops alternately (Figure \ref{fig:quasi-flat} (iii)).
Recall that this point is very delicate and important in view of stability --- 
in fact, by \cite[Section 6.3]{MY_Crelle} if either
\begin{itemize}
    \item $L_1=0$ or $L_{N+1}=0$ (Figure \ref{fig:quasi-flat} (i)), or
    \item there is $1<j<N+1$ such that $L_j=0$ and $\sigma_{j-1} \neq \sigma_j$ (Figure \ref{fig:quasi-flat} (ii)),
\end{itemize}
then the corresponding curve is unstable under the pinned boundary condition.

\begin{remark}\label{rem:threshold-stability}
The equality case $\frac{L}{p-1}=|P_1-P_0|$ is slightly delicate in view of stability.
In this case, the set $\Ap$ does admit flat-core pinned $p$-elasticae, but each of them needs to have a loop touching an endpoint, thus unstable by \cite{MY_Crelle}*{Proposition 6.3}.
This with \cite{MY_Crelle}*{Corollaries 2.12 and 2.14} implies Theorem \ref{thm:previous_rigidity}.
\end{remark}

\begin{remark}\label{rem:simplex}
    Relation \eqref{eq:sum-flatparts} forms an $N$-simplex.
    Hence if $|P_0-P_1|>\frac{L}{p-1}$, the space of flat-core pinned $p$-elasticae in $\Ap$ can be written as a disjoint union $\bigcup_{N=1}^\infty E_N$, where each $E_N$ is isomorphic to $\{-1,1\}^N\times\Delta^N$.
    Here $\Delta^N$ stands for the standard $N$-simplex.
    From this point of view, the alternating class corresponds to the interior part of $\Delta^N$.
    The quasi-alternating class additionally contains a part of the boundary $\partial\Delta^N$.
    The remaining part consists of unstable solutions.
\end{remark}


\section{Hooked $p$-elasticae}\label{sect:hooked}

Given $p\in(1,\infty)$ and $0<\ell<L$, we define a class of curves subject to a free boundary condition, which we call the \emph{hooked boundary condition}, by 
\begin{align*}
    \begin{split}
            \Ah &= \Ah(\ell,L)\\
            &:=\Set{ \gamma \in W^{2,p}_\mathrm{arc}(0,L; \R^2) | (\gamma(L)-\gamma(0))\cdot e_1 =\ell,\ \gamma'(L)=-e_1
            }.
    \end{split}
\end{align*}
As explained in the introduction, we will decompose an alternating flat-core $p$-elastica into a finite family of curves in $\Ah$ in order to reduce our stability problem to (global) minimization problems for those curves.

In this paper we define hooked $p$-elasticae as follows: 

\begin{definition}[Hooked $p$-elastica]\label{def:hooked-p-elastica*}
    We call $\gamma \in \Ah$ a \emph{hooked $p$-elastica} if $\gamma$ is a $p$-elastica with curvature $k$ such that the function $w:=|k|^{p-2}k\in C^2([0,L])$ satisfies $w(0)=w'(L)=0$.
\end{definition}

The above definition is, as in the pinned case, the natural first-order necessary condition to be a (local) minimizer of $\mathcal{B}_p$ in $\Ah$.
We postpone detailed arguments for this fact to Appendix~\ref{sect:appendix-A} since we can argue similarly to our previous study of pinned $p$-elasticae \cite{MR4852697} up to minor modifications.
However, we stress that the condition $w'(L)=0$ appears (not in the pinned case but) only in the hooked case and makes the derivation slightly more delicate, but plays a very important role in our reduction process.

In the following, we first classify all the possible hooked $p$-elasticae in Section~\ref{sect:classify-hooked}. 
We then prove unique existence of minimal hooked $p$-elasticae and obtain the explicit minimal energy in Section~\ref{sect:minimial-hook}.

\subsection{Classification for hooked $p$-elasticae}\label{sect:classify-hooked}

In this subsection we will deduce classification of hooked $p$-elasticae. 
Combining definition of hooked $p$-elasticae with our previous classification result, we obtain the following 

\begin{lemma}\label{lem:hooked-dichotomy}
Let $\gamma \in \Ah$ be a hooked $p$-elastica.
Then $\gamma$ is either a wavelike $p$-elastica or a flat-core $p$-elastica.
In addition, the signed curvature $k$ satisfies the following additional boundary condition:
\begin{align}\label{eq:additional-BC}
    k(0)=0, \quad k(L)\neq0, \quad k'(L)=0.
\end{align}
\end{lemma}

\begin{remark}[Differentiability of the curvature]\label{rem:diff_curvature}
By Proposition~\ref{prop:regularity-w}, the signed curvature is differentiable whenever $k\neq0$, and hence $k'(L)$ makes sense in \eqref{eq:additional-BC}.
\end{remark}

\begin{proof}[Proof of Lemma~\ref{lem:hooked-dichotomy}]
By $w(0)=0$ in Definition \ref{def:hooked-p-elastica*} we deduce that $k(0)=0$.
This with Proposition~\ref{prop:MY2203-thm1.1} implies that $\gamma$ is either linear, wavelike, or flat-core.
In addition, the linear case is clearly ruled out by the hooked boundary condition.

We next prove $k(L)\neq0$ by contradiction, so suppose $k(L)=0$. 
This assumption together with Definition \ref{def:hooked-p-elastica*} implies $w(L)=w'(L)=0$.
Recall that $w$ satisfies \eqref{eq:0610-1} in the classical sense.
Then, by the known classification \cite[Theorem 4.1]{MY_AMPA} on the corresponding Cauchy problem for equation \eqref{eq:0610-1},
we see that either $k\equiv0$ or $k$ is of flat-core type, i.e., of the form \eqref{eq:flat-type}.
Since $\gamma$ is not linear, it is flat-core.
Therefore, it suffices to check that there is no flat-core $p$-elastica satisfying both $k(0)=k(L)=0$ and the hooked boundary condition.
This follows by our previous classification of pinned $p$-elasticae \cite[Theorem 1.1]{MR4852697}; in fact, if $k(0)=k(L)=0$ and $\gamma$ is a flat-core $p$-elastica, then up to a similar transformation the curve $\gamma$ is of the form \eqref{eq:N-loop-flat-core}.
Hence by \eqref{eq:flcr-property} the vectors $\gamma'(L)$ and $\gamma(L)-\gamma(0)$ are in the same direction, which contradicts our hooked boundary condition.
This implies $k(L)\neq0$.

Finally, by Remark \ref{rem:diff_curvature}, we now have
$w'(L)=(p-1)|k(L)|^{p-2}k'(L)$, which together with the fact that $w'(L)=0$ implies $k'(L)=0$.
The proof is complete.
\end{proof}

Now we go into more details.
First we consider the wavelike case in Lemma~\ref{lem:hooked-dichotomy}. 
Recall from Proposition~\ref{prop:MY2203-thm1.1} that the curvature of wavelike $p$-elasticae is given in terms of $\cn_p$.
In order to characterize the points where $k'$ vanishes, we prepare the following

\begin{lemma}\label{lem:zero-cn'_p}
Let $q\in(0,1)$.
Then
\begin{align*}
\Set{ x\in\R | \tfrac{\partial}{\partial x} \cn_p(x,q)=0 } =
    \begin{cases}
    \Z\K_{1,p}(q), \quad &\text{if} \ \ p\in(1,2),\\
    2\Z\K_{1,p}(q), &\text{if} \ \ p\in[2,\infty).
    \end{cases}
\end{align*}
\end{lemma}
\begin{proof}
By definition of $\cn_p$
we obtain
\[
\frac{\partial}{\partial x}\cn_p(x,q)=-\frac{2}{p}|\cos \amcn(x,q)|^{\frac{4}{p}-2} \sin \amcn(x,q)\sqrt{1-q^2\sin^2\amcn(x,q)}
\]
for $x\in \R\setminus Z_{p,q}$, where $Z_{p,q}$ is given by \eqref{eq:Z_pq}.
Since $\frac{4}{p}-2>0$ is equivalent to $p<2$, if  $p\in(1,2)$, then $ \tfrac{\partial}{\partial x}\cn_p(x,q)$ is well defined as well as $x\in Z_{p,q}$.
Thus we see that $\tfrac{\partial}{\partial x}\cn_p(x,q)=0$ holds if and only if
\begin{align}\label{eq:cond-cn'_is_0}
\begin{cases}
    \text{$\cos\amcn(x,q)=0$ \ \ \text{or} \ \  $\sin\amcn(x,q)=0$} \ \ &\text{for } \ p\in(1,2),\\
   \text{$\sin\amcn(x,q)=0$} &\text{for } \ p\in[2,\infty).
\end{cases}
\end{align}
Since $\amcn$ is the inverse of the strictly increasing and periodic function $\Fcn$, cf.\ \eqref{eq:period_Ecn}, for any $n\in\Z$ and $x\in \R$,
\begin{align} \label{eq:0925-7}
\amcn(x + 2n \K_{1,p}(q), q) = \amcn(x, q) + n\pi.
\end{align}
This together with $\amcn(\Kcn(q),q)=\pi/2$ and $\amcn(0,q)=0$ implies that
\begin{align*}
\Set{ x\in\R | \cos\amcn(x,q)=0  } &= (2\Z+1)\Kcn(q),\\
\Set{ x\in\R | \sin\amcn(x,q)=0  } &= 2\Z\Kcn(q).
\end{align*}
This with \eqref{eq:cond-cn'_is_0} completes the proof.
\end{proof}

This lemma together with \eqref{eq:Z_pq} directly implies the following

\begin{corollary}\label{cor:cn_p-diff-0}
    Let $p\in(1,\infty)$, $q\in(0,1)$, and $x\in\R$.
    Then
    $\cn_p(x,q)\neq0$ and $\tfrac{\partial}{\partial x}\cn_p(x,q)=0$ hold simultaneously if and only if $x\in 2\Z\Kcn(q)$.
\end{corollary}

Now we turn to the flat-core case in Lemma~\ref{lem:hooked-dichotomy}. 
In view of $k(L)\neq0$, any hooked flat-core $p$-elastica has a loop part around $s=L$. 
Recall from Proposition~\ref{prop:MY2203-thm1.1} that the curvature of the loop part is given in terms of $\sech_p|_{[-\K_p(1), \K_p(1)]}$.
In order to characterize the property that $k(L)\neq 0$ and $k'(L)=0$ in \eqref{eq:additional-BC}, we prepare the following

\begin{lemma}\label{lem:zero-sech'_p}
Let $p>2$ and 
$x\in (-\K_p(1), \K_p(1))$. 
Then, $\sech_p'{x}=0$ holds if and only if $x=0$.
\end{lemma}
\begin{proof}
Note that the differentiability of $\sech_p$ on $\R\setminus\{\pm\K_p(1)\}$ is already shown in \cite[Theorem 3.16]{MY_AMPA}.
Let $x\in (-\K_p(1), \K_p(1))$.
By direct computation we have
\[
\sech_p'{x}=-\frac{2}{p}\big(\cos\amcn(x,1)\big)^{\frac{4}{p}-1 } \sin \amcn(x,1),
\]
and also $|\amcn(x,1)|<\pi/2$, so that $\sech_p'{x}=0$ if and only if $x=0$.
\end{proof}

Now we are in a position to state the classification for hooked $p$-elasticae with an explicit parametrization in terms of $\gamma_w$, $\gamma^\pm_b$, and $\gamma_\ell$ introduced in Proposition~\ref{prop:MY2203-thm1.1}.

\begin{theorem}[Classification of hooked $p$-elasticae]\label{thm:classify-hook}
Let $p\in(1,\infty)$ and $0<\ell<L$.
Let $\gamma \in \Ah$ be a hooked $p$-elastica.
\begin{itemize}
    \item[(i)] If $p\in(1,2]$, or if $p\in(2,\infty)$ and $\ell\in(0,\frac{1}{p-1}L)$, then, up to vertical translation and reflection,  $\gamma$ is given by
    \begin{align}\label{eq:hooked-wave}
    \gamma(s)&= \tfrac{1}{\alpha_n} R_\pi \Big( \gamma_w(\alpha_n s + \Kcn(q),q) -\gamma_w(\Kcn(q),q) \Big), \quad s\in[0,L], \\
    \alpha_n&:=\frac{(2n-1)\K_{1,p}(q)}{L}, \notag
    \end{align}
    for some $n\in \N$, where $q\in(0,1)$ is a unique solution of 
    \begin{align}\label{eq:wave-loop-r}
      2\frac{\E_{1,p}(q)}{\K_{1,p}(q)} -1 = -\frac{\ell}{L}.  
    \end{align}
    \item[(ii)] If $p\in(2,\infty)$ and $\ell\in[\frac{1}{p-1}L,L)$, then, up to vertical translation (and reflection), $\gamma$ is given by
    \begin{align}
    \gamma(s)&=\tfrac{1}{\bar\alpha_n}R_\pi\,\Gamma_n\left(\bar\alpha_n s \right), \quad s\in[0,L], \label{eq:hooked-flat}\\
    \bar\alpha_n &:= (2n-1) \frac{1}{L-\ell} \frac{p-2}{p-1} \K_p(1), \label{eq:hooked-flat-alpha}
    \end{align}
    for some $n\in \N$, where $\Gamma_n$ is an arclength parametrized curve defined by 
       \begin{align}
       \begin{split}\label{eq:def-Gamma_n}
       \Gamma_1&:=\gamma_{\ell}^{L_1}\oplus \left(\gamma_{b}^{\sigma_1}\big|_{[-\K_p(1), 0]}\right), 
       \\
       \Gamma_n&:=\Bigg( \bigoplus_{j=1}^{n-1} \big( \gamma_{\ell}^{L_j} \oplus \gamma_{b}^{\sigma_j} \big) \Bigg) \oplus \gamma_{\ell}^{L_{n}}\oplus \left(\gamma_{b}^{\sigma_n}\big|_{[-\K_p(1), 0]}\right) \quad\text{if}\quad  n\geq2, 
       \end{split}
       \end{align}
   for some $\sigma_1,\dots,\sigma_n\in\{+,-\}$ and
   $L_1,\dots,L_{n} \geq0$ such that
    \begin{align} \label{eq:sum-flatparts-Ah}
    \sum_{j=1}^{n}L_j = (2n-1)\frac{\frac{\ell}{L}-\frac{1}{p-1}}{1-\frac{\ell}{L}} \K_{p}(1).
    \end{align}
\end{itemize}
\end{theorem}

\begin{proof}
    Note first that $\gamma$ is either wavelike or flat-core, by Proposition \ref{prop:MY2203-thm1.1} and Lemma \ref{lem:hooked-dichotomy}.
    In what follows, we mainly prove the following propositions.
    \begin{itemize}
        \item (Case 1) $\gamma$ is wavelike if and only if the assertion in (i) holds.
        \item (Case 2) $\gamma$ is flat-core if and only if the assertion in (ii) holds.
    \end{itemize}
    Along the way, we also observe that Case 1 occurs if and only if $\frac{\ell}{L}<\frac{1}{p-1}$ (i.e., either $p\in(1,2]$, or $p\in(2,\infty)$ and $\ell\in(0,\frac{1}{p-1}L)$).
    This automatically means that Case 2 occurs if and only if $\frac{\ell}{L}\geq\frac{1}{p-1}$ (i.e., $p\in(2,\infty)$ and $\ell\in[\frac{1}{p-1}L,L)$).

    \textbf{Case 1} (\textsl{Wavelike $p$-elasticae})\textbf{.}
    Let $\gamma\in\Ah$ be a wavelike hooked $p$-elastica.
    Up to a vertical translation, we may assume that $\gamma(0)=(0,0)$.
    Then Proposition \ref{prop:MY2203-thm1.1} implies that $\gamma$ is given by 
    \begin{equation}\label{eq:0710-11}
        \gamma(s)=\frac{1}{\alpha}AR_\phi(\gamma_w(\alpha s+s_0,q)-\gamma_w(s_0))
    \end{equation}
    for some $q\in(0,1)$, $s_0\in\R$, $\alpha>0$, $\phi\in[0,2\pi)$, and 
    $A\in\{I,J\}$, where $I$ denotes the identity and $J$ denotes the vertical reflection $P\mapsto P-2(P\cdot e_2)e_2$, both given by $2\times2$ matrices.
    By Proposition \ref{prop:MY2203-thm1.1}, the curvature of $\gamma$ is of the form $k(s)=\pm2\alpha q\cn_p(\alpha s+s_0,q)$.
    Now we use the boundary conditions in Lemma \ref{lem:hooked-dichotomy}.
    Since $k(0)=0$, by \eqref{eq:Z_pq} we have $s_0\in(2\Z+1)\Kcn(q)$; by periodicity we may assume that $s_0\in\{\Kcn(q),-\Kcn(q)\}$; by symmetry, up to reflection (i.e., changing $A$ if necessary) we may eventually assume that
    \begin{equation}\label{eq:0710-12}
        s_0=\Kcn(q).
    \end{equation}
    Moreover, since $k(L)\neq0$ and $k'(L)=0$, by Corollary \ref{cor:cn_p-diff-0} we have $\alpha L+\Kcn(q) \in 2\Z\Kcn(q)$, or equivalently $\alpha L \in (2\Z-1)\Kcn(q)$.
    Since $\alpha,L>0$, this means that there is some $n\in\N$ such that 
    \begin{equation}\label{eq:0710-13}
        \alpha = \frac{(2n-1)\Kcn(q)}{L}\ (=\alpha_n).
    \end{equation}
    By \eqref{eq:0710-12}, \eqref{eq:0710-13}, and the formula for $\theta_w$ in Proposition \ref{prop:MY2203-thm1.1}, we also deduce that 
    $$|\theta(L)|= \phi+\theta_w(\alpha L+s_0) =  \phi+2\arcsin(q\sn_p(2n\Kcn(q),q)) = \phi \quad (\text{mod $2\pi\Z$}).$$
    Since $\gamma'(L)= -e_1$ by definition of $\Ah$, we need to have
    \begin{equation}\label{eq:0710-14}
        \phi = \pi.
    \end{equation}
    The necessary conditions in \eqref{eq:0710-11}--\eqref{eq:0710-14} already imply \eqref{eq:hooked-wave}.
    Moreover, in addition to \eqref{eq:0710-11}--\eqref{eq:0710-14}, by using $(\gamma(L)-\gamma(0))\cdot e_1=\ell$ and also formula \eqref{eq:EP2}, we deduce that
    \begin{align*}
    -\frac{L}{(2n-1)\Kcn(q)}\Big( 2\Ecn(\amcn(2n\Kcn(q),q),q)-2n\Kcn(q)& \\
    -2\Ecn(\amcn(\Kcn(q),q),q) +\Kcn(q)& \Big) =\ell.
    \end{align*}
    By $\amcn(\Kcn(q),q)=\frac{\pi}{2}$, \eqref{eq:period_Ecn}, and \eqref{eq:0925-7}, we thus find that \eqref{eq:wave-loop-r} is also necessary to hold.
    By Lemma~\ref{lem:nabe-lem2}, equation \eqref{eq:wave-loop-r} has a solution $q\in(0,1)$ (if and) only if $\frac{\ell}{L}<\frac{1}{p-1}$, and such a solution is unique if exists.
    In summary, if $\gamma$ is a wavelike $p$-elastica, then the assertion in (i) holds true, and also necessarily $\frac{\ell}{L}<\frac{1}{p-1}$.
    
    Conversely, if the assertion in (i) holds true, then it is clear that $\gamma$ needs to be a wavelike $p$-elastica, while it is also necessary that $\frac{\ell}{L}<\frac{1}{p-1}$ by Lemma~\ref{lem:nabe-lem2}.

    \textbf{Case 2} (\textsl{Flat-core $p$-elasticae})\textbf{.} 
    Let  $\gamma \in \Ah$ be a flat-core hooked $p$-elastica. 
    Up to translation, we may assume that $\gamma(0)=(0,0)$.
    Then, by Proposition~\ref{prop:MY2203-thm1.1}, 
    \begin{align}\label{eq:0803-1}
    \gamma(s)=\frac{1}{\alpha}R_\phi \big( \gamma_f(\alpha s+s_0) - \gamma_f(s_0) \big),\quad \gamma_f=\bigoplus_{j=1}^m\big(\gamma_\ell^{L_j}\oplus\gamma_b^{\sigma_j}\big),
    \end{align}
    for some $m\in \N$, $s_0\in\R$, $\alpha>0$, $\phi\in[0,2\pi)$, $\boldsymbol{\sigma}=(\sigma_1,\ldots, \sigma_m) \in \{+,-\}^m$, and $\boldsymbol{L}=(L_1,\ldots, L_m) \in [0,\infty)^m$.
    (The vertical reflection corresponds to the replacement of $\boldsymbol{\sigma}$ with $-\boldsymbol{\sigma}:=(-\sigma_1,\ldots,-\sigma_m)$.)
    Now we use the boundary condition in Lemma~\ref{lem:hooked-dichotomy}.
    Since the curvature of $\gamma$ vanishes at $s=0$, by changing $m$, $\boldsymbol{\sigma}$, and $\boldsymbol{L}$ if necessary, we may assume that $0\leq s_0\leq L_1$, and then, by replacing $L_1$ with $L_1- s_0$, we may assume that $s_0=0$ in \eqref{eq:0803-1}.
    Thus the form \eqref{eq:0803-1} can be reduced to 
    \begin{align}\label{eq:0803-10}
    \gamma(s)=\frac{1}{\alpha}R_\phi \gamma_f(\alpha s),\quad \gamma_f=\bigoplus_{j=1}^m\big(\gamma_\ell^{L_j}\oplus\gamma_b^{\sigma_j}\big).
    \end{align}
    By Proposition~\ref{prop:MY2203-thm1.1}, the curvature of $\gamma$ is given by $k(s)=\alpha k_f(\alpha s)$, where
    \[
    k_f(s)=\sum_{j=1}^m 2\sigma_j \sech_p(s-s_j), \quad  s_j:=(2j-1)\K_p(1)+\sum_{i=1}^jL_i.
    \]
    Since $k(L)\neq0$ and $k'(L)=0$, we deduce from the above form combined with Lemma~\ref{lem:zero-sech'_p} that $\alpha L=s_n$ for some $n\in\{1,\ldots,m\}$, which is equivalent to
    \begin{align}\label{eq:0803-2}
        \alpha=\frac{1}{L}\Big((2n-1)\K_p(1) +\sum_{j=1}^n L_j \Big)
    \end{align}
    for some $n\in\{1,\ldots,m\}$.
    This also means that $\gamma_f$ in \eqref{eq:0803-10} satisfies $\gamma_f(\alpha L)=\gamma_b^{\sigma_n}(0)$, and hence $\gamma_f$ coincides with $\Gamma_n$ defined by \eqref{eq:def-Gamma_n}.
    Next we use the original hooked boundary condition.
    Since $\alpha L=s_n$ implies $(\gamma_f)'(\alpha L)=(\gamma_b^{\sigma_n})'(0)$, and since $(\gamma_b^{\sigma_n})'(0)=e_1$ by Remark~\ref{rem:loop} (and Figure~\ref{fig:gamma_b_v2}), we deduce from \eqref{eq:0803-10} that
    \begin{align*}
        \gamma'(L) = R_\phi (\gamma_b^{\sigma_n})'(0) = R_\phi e_1.
    \end{align*}
    On the other hand, $\gamma'(L)=-e_1$ by definition of $\Ah$, and hence
    we need to have
    \begin{align}\label{eq:0803-3}
        \phi=\pi.
    \end{align}
    By Remark \ref{rem:loop} and symmetry (cf.\ Figure \ref{fig:gamma_b_v2}) we have $\gamma_b^{\pm}(\K_p(1))-\gamma_b^{\pm}(-\K_p(1))=-\frac{2}{p-1}\K_p(1)e_1$ and $\gamma_b^{\pm}(0)-\gamma_b^{\pm}(-\K_p(1))=-\frac{1}{p-1}\K_p(1)e_1$.
    Combining these with $(\gamma(L)-\gamma(0))\cdot e_1=\ell$ from definition of $\Ah$, we deduce that
    \begin{align}\label{eq:0803-4}
    \frac{1}{\alpha}\Big(\frac{2n-1}{p-1}\K_p(1)+\sum_{j=1}^n L_j \Big)=\ell.
    \end{align}
    Combining this with \eqref{eq:0803-2}, we find that $\alpha$ in \eqref{eq:0803-10} is equal to $\bar{\alpha}_n$ defined by \eqref{eq:hooked-flat-alpha}, and we need to have \eqref{eq:sum-flatparts-Ah}.
    Consequently, the necessary conditions in \eqref{eq:0803-1}--\eqref{eq:0803-4} imply \eqref{eq:hooked-flat} and \eqref{eq:hooked-flat-alpha}.
    Now we find that if $\gamma$ is a flat-core $p$-elastica, then the assertion in (ii) holds true.
    
    As in Case 1, the converse is obvious.
    (We can also directly observe that $\frac{L}{p-1} \leq \ell$ is necessary and sufficient for Case 2, but logically we need not verify it.)
    The proof is complete.
\end{proof}

\subsection{Global minimizers}\label{sect:minimial-hook}

Thanks to Theorem~\ref{thm:classify-hook} we can detect global minimizers by comparing the energy of all hooked $p$-elasticae. 
We first prove the existence of global minimizers by the standard direct method.

\begin{proposition}\label{prop:direct-method-Ah}
Given $p\in(1,\infty)$ and $0<\ell<L$, there exists a solution to the following minimization problem 
\[\min_{\gamma\in\Ah} \mathcal{B}_p[\gamma].\]
\end{proposition}
\begin{proof}
Let $\{\gamma_j\}_{j\in \N}\subset \Ah$ be a minimizing sequence of $\mathcal{B}_p$ in $\Ah$, i.e.,
\begin{align} \label{eq:min_seq}
 \lim_{j\to\infty}\mathcal{B}_p[\gamma_j]=\inf_{\gamma\in \Ah}\mathcal{B}_p[\gamma]. 
\end{align}
We may suppose that, up to translation, $\gamma_j(0)=(0,0)$.
By \eqref{eq:min_seq}, there is $C>0$ such that $\mathcal{B}_p[{\gamma}_j]\leq C$, and this together with the fact that $\|\gamma_j''\|_{L^p}^p =\mathcal{B}_p[{\gamma}_j]$ yields the uniform estimate of $\|\gamma_j''\|_{L^p}$.
Using $|\gamma_j'|\equiv 1$ and $\gamma_j(0)=(0,0)$, we also obtain the bounds on the $W^{1,p}$-norm.
Therefore, $\{\gamma_j\}_{j\in \N}$ is uniformly bounded in $W^{2,p}(0,L;\R^2)$ so that there is a subsequence (without relabeling) that converges in the sense of $W^{2,p}$-weak and $C^1$ topology.
Thus the limit curve $\gamma_\infty$ satisfies $\gamma_\infty \in W^{2,p}(0,L;\R^2)$, $|\gamma_\infty'|\equiv 1$, $(\gamma_\infty(L)-\gamma_\infty(0))\cdot e_1=\ell$, and 
$\gamma_\infty'(L)=-e_1$, 
which implies that $\gamma_\infty \in \Ah$.
In addition, since $\gamma_\infty$ is parametrized by its arclength, the weak lower semicontinuity for $\|\cdot\|_{L^p}$ ensures that 
\[
\mathcal{B}_p[\gamma_\infty]=\|\gamma_\infty''\|_{L^p}^p \leq \liminf_{j\to\infty} \|\gamma_j''\|_{L^p}^p=\liminf_{j\to\infty}\mathcal{B}_p[\gamma_j].
\]
Thus we see that $\gamma_\infty$ is a minimizer of $\mathcal{B}_p$ in $\Ah$.
\end{proof}

We also recall the following lemma to calculate the $p$-bending energy of hooked $p$-elasticae. 

\begin{lemma}[{\cite[Lemma 4.2]{MR4852697}}] \label{lem:int-cn_p}
For each $q\in(0,1)$ (including $q=1$ if $p>2$),  
\begin{align} \notag
    \int_0^{\K_{1,p}(q)} |\cn_p(s, q)|^p\,ds 
= \frac{1}{q^2}\E_{1,p}(q) + \Big( 1-\frac{1}{q^2}\Big)\K_{1,p}(q).
\end{align}
\end{lemma}

We now prove the main theorem for global minimizers.

\begin{theorem}[Minimal hooked $p$-elasticae]\label{thm:min_free_half}
    Let $p\in(1,\infty)$, $0<\ell<L$, and $\gamma\in\Ah$.
\begin{itemize}
    \item[(i)] If $p\in(1,2]$, or if $p\in(2,\infty)$ and $\ell\in(0,\frac{1}{p-1}L)$, then for the unique solution $q\in(0,1)$ to \eqref{eq:wave-loop-r}, 
    \begin{equation}\label{eq:minimal-loop}
        \B_p[\gamma] \geq (2q)^p\Kcn(q)^{p-1}\left( \frac{1}{q^2}\Ecn(q)+\left(1-\frac{1}{q^2}\right)\Kcn(q) \right)\frac{1}{L^{p-1}},
    \end{equation}
    where equality holds if and only if $\gamma$ is given by \eqref{eq:hooked-wave} with $n=1$, up to vertical translation and reflection.
    \item[(ii)] If $p\in(2,\infty)$ and $\ell\in[\frac{1}{p-1}L,L)$, then
    \begin{equation}\label{eq:minimal-flat}
        \B_p[\gamma] \geq 2^p\Kcn(1)^{p-1}\Ecn(1) \left(\frac{p-2}{p-1} \right)^{p-1}\frac{1}{(L-\ell)^{p-1}},
    \end{equation}
    where equality holds if and only if $\gamma$ is given by \eqref{eq:hooked-flat} with $n=1$, up to vertical translation (and reflection).
\end{itemize}
\end{theorem}

\begin{proof}
The existence of minimizers follows from Proposition~\ref{prop:direct-method-Ah}.
Fix any minimizer $\gamma \in \Ah$. 
Then $\gamma$ is a hooked $p$-elastica (cf.\ Appendix~\ref{sect:appendix-A}).
We divide the proof into two cases along the classification for hooked $p$-elasticae in Theorem~\ref{thm:classify-hook}. 

First we consider case (i).
In this case, up to translation and reflection, $\gamma$ is given by \eqref{eq:hooked-wave} for some $n\in \N$, and  the signed curvature $k$ of $\gamma$ is 
\[
k(s)=2q\alpha_n\cn_p \big(\alpha_n s+\K_{1,p}(q), q\big), \quad s\in[0,L]. 
\]
Therefore we have
\begin{align*}
\mathcal{B}_p[\gamma] 
&=(2q)^p\alpha_n^{p-1}(2n-1)\int_0^{\K_{1,p}(q)}|\cn_p(x,q)|^p\,dx,
\end{align*}
where we used the symmetry and periodicity of $\cn_p$ in Proposition~\ref{prop:property-cndn}. 
The case $n=1$ corresponds to a unique minimizer, and Lemma~\ref{lem:int-cn_p} implies \eqref{eq:minimal-loop}.

Next we address case (ii).
In this case, up to a vertical translation (and reflection), the curve $\gamma$ is given by \eqref{eq:hooked-flat} with \eqref{eq:hooked-flat-alpha} for some $n\in \N$.
Then, the signed curvature $k$ of $\gamma$ is 
\[
k(s)=\bar\alpha_n k_f(\bar\alpha_n s), \quad k_f(s)=\sum_{j=1}^n 2\sigma_j \sech_p(s-s_j),
\]
where $s_j:=(2j-1)\K_p(1)+\sum_{i=1}^j L_i$.
Since $s_n=\bar\alpha_n L$ holds as in the proof of Theorem~\ref{thm:classify-hook}, we have 
\begin{align*}
\mathcal{B}_p[\gamma]
=& 
\sum_{j=1}^{n-1}\int_{s_j-\K_p(1)}^{s_j+\K_p(1)}2^p\bar\alpha_n^{p-1} |\sech_p(s-s_j)|^p\,ds \\
&\qquad+\int_{s_n-\K_p(1)}^{s_n}2^p\bar\alpha_n^{p-1} |\sech_p(s-s_n)|^p\,ds 
\end{align*}
(the first sum is interpreted as $0$ if $n=1$).
From the symmetry and periodicity of $\sech_p$ in Proposition~\ref{prop:property-cndn} we deduce that 
\begin{align*}
\mathcal{B}_p[\gamma] =2^p \bar\alpha_n^{p-1} (2n-1) \int_0^{\K_p(1)}|\sech_p{s}|^p\,ds.
\end{align*}
The case $n=1$ corresponds to a unique minimizer, and Lemma~\ref{lem:int-cn_p} (with $q=1$) implies \eqref{eq:minimal-flat}.
\end{proof}

\begin{remark}\label{rem:min_free_half}
    By symmetry the same result also holds for $\Ah$ replaced with
    \begin{align*}
        \begin{split}
            \Ah' &= \Ah'(\ell,L)\\
            &:=\Set{ \gamma \in W^{2,p}_\mathrm{arc}(0,L; \R^2) | (\gamma(L)-\gamma(0))\cdot e_1 =\ell,\ \gamma'(0)=-e_1 }.
        \end{split}
    \end{align*}
\end{remark}

\begin{remark}
    For the proof of Theorem \ref{thm:alternating_stability} (Theorem \ref{thm:alternating_detailed}) in the next section, we will use Theorem \ref{thm:min_free_half} only in the case that $p\in(2,\infty)$ and $\ell\in[\frac{1}{p-1}L,L)$.
    However, our full classification results in this section would be useful to highlight the idiosyncrasy of the degenerate case under consideration.
\end{remark}

\section{Stability of alternating flat-core $p$-elasticae}\label{sect:proof-MainResults}

In this section, we prove the desired stability of alternating flat-core $p$-elasticae.
More precisely, by applying Theorem~\ref{thm:min_free_half} we prove the following

\begin{theorem}\label{thm:alternating_detailed}
    Let $p\in(2,\infty)$, $P_0, P_1\in \R^2$, and $L>0$ such that $\frac{L}{p-1} < |P_1-P_0|<L$.
    Let $N\in\N$ and $\gamma\in \Ap(P_0, P_1, L)$ be an $N$-loop alternating flat-core $p$-elastica (see Definition \ref{def:alternating}).
    Then $\gamma$ is a local minimizer of $\mathcal{B}_p$ in $\Ap(P_0, P_1, L)$.
\end{theorem}

The following is a key lemma deduced from Theorem~\ref{thm:min_free_half}. 

\begin{lemma}\label{lem:relaxation}
    Let $p\in(2,\infty)$, $L>0$, $\ell\in(\frac{1}{p-1}L,L)$, and $N$ be a positive integer.
    Let $\{\gamma_i\}_{i=1}^N$ be an $N$-tuple of curves $\gamma_i\in \Ah(\ell_i,L_i)\cup\Ah'(\ell_i,L_i)$ with $L_i>0$ and $\ell_i\in[\frac{1}{p-1}L_i,L_i)$ such that $\sum_{i=1}^N L_i = L$ and $\sum_{i=1}^N\ell_i = \ell$.
    Then
    \begin{equation*}
        \sum_{i=1}^N \B_p[\gamma_i] \geq \frac{C_pN^p}{(L-\ell)^{p-1}},
    \end{equation*}
    where $C_p:=2^p\Kcn(1)^{p-1}\Ecn(1)(\frac{p-2}{p-1})^{p-1}$.
\end{lemma}

\begin{proof}
    For each $i$ we apply Theorem \ref{thm:min_free_half} or Remark \ref{rem:min_free_half} to deduce that
    \begin{equation*}
        \sum_{i=1}^N\B_p[\gamma_i]\geq C_p\sum_{i=1}^{N}(L_i-\ell_i)^{1-p}.
    \end{equation*}
    Then Jensen's inequality applied to the convex function $x\mapsto x^{1-p}$ implies that 
    $$C_p\sum_{i=1}^{N}(L_i-\ell_i)^{1-p} \geq C_pN \left( \frac{1}{N}\sum_{i=1}^{N}(L_i-\ell_i) \right)^{1-p}= C_pN^p(L-\ell)^{1-p},$$
    where in the last equality we have used $\sum_{i=1}^NL_i=L$ and $\sum_{i=1}^N\ell_i = \ell$.
\end{proof}

We are now in a position to complete the proof of Theorem~\ref{thm:alternating_detailed}.

\begin{proof}[Proof of Theorem~\ref{thm:alternating_detailed}]
    Fix any alternating flat-core $p$-elastica $\gamma$.
    Up to similarity and reparametrization, we may only consider the case that
    \begin{align}\label{eq:alternating-model}
        \gamma = R_\pi \bigg( \Big( \bigoplus_{j=1}^N \big(\gamma_\ell^{L_j} \oplus \gamma_b^{\sigma_j}\big) \Big) \oplus \gamma_\ell^{L_{N+1}} \bigg),
    \end{align}
    for some $N\in \N$, $\{\sigma_j\}_{j=1}^N \subset\{+,- \}$, and $L_1, \ldots, L_{N+1}>0$.
    In this case, we have $\gamma\in \Ap(P_0,P_1, L)$ with $P_0:=(0,0)$, $P_1:=(\ell,0)$, $\ell:= \frac{2N}{p-1}\K_p(1)+\sum_{j=1}^{N+1} L_j$, and $L:=2N\K_p(1) + \sum_{j=1}^{N+1} L_j$.
    Since the signed curvature $k$ of $\gamma$ is given by $k(s)=\sum_{j=1}^N 2\sigma_j \sech_p(s-s_j)$ with $s_j:=(2j-1)\K_p(1)+\sum_{i=1}^j L_i$, we can explicitly compute (using Lemma~\ref{lem:int-cn_p}) 
    \begin{align}\label{eq:0710-9}
    \begin{split}
        \B_p[\gamma]&= \sum_{j=1}^{N}\int_{s_j-\K_p(1)}^{s_j+\K_p(1)} \big(2\sech_p (s-s_j)\big)^p\,ds \\
        &= 2^{p+1}N \E_{1,p}(1) = C_p\frac{(2N)^p}{(L-\ell)^{p-1}},
    \end{split}
    \end{align}
    where $C_p=2^p\Kcn(1)^{p-1}\Ecn(1)(\frac{p-2}{p-1})^{p-1}$ as in Lemma \ref{lem:relaxation}.
    
    Now we prove that the above 
    $\gamma$ is a local minimizer of $\B_p$ in $\Ap$.
    Take an arbitrary sequence $\{\gamma_\nu\}_{\nu\in\N}\subset\Ap$ such that $\gamma_\nu\to\gamma$ in $W^{2,p}(0,L;\R^2)$, and hence also in $C^1([0,L];\R^2)$.
    It suffices to prove that $\B_p[\gamma_\nu]\geq\B_p[\gamma]$ holds for all large $\nu$.

    Choose a partition $\{s_i\}_{i=1}^{2N+1}$ of $[0,L]$ as follows (see Figure~\ref{alternating-partition}).
    Let $s_1:=0$, $s_{2N+1}:=L$, $s_{2i}:=(2i-1)\K_p(1)+\sum_{m=1}^{i}L_m$ for $i\in\{1,\dots,N\}$, which corresponds to the midpoint of the $i$-th loop $\gamma_b^{\sigma_i}$, and $s_{2i-1}:=2(i-1)\K_p(1)+\sum_{m=1}^{i-1}L_m+\frac{1}{2}L_{i}$ for $i\in\{2,\dots,N\}$, which corresponds to the midpoint of the $i$-th segment $\gamma_\ell^{L_i}$.
    Let
    \begin{equation}\label{eq:0710-2}
    \tilde{L}_i:=s_{i+1}-s_{i}, \quad \tilde{\ell}_i:=\big(\gamma(s_{i+1})-\gamma(s_{i})\big)\cdot e_1 \quad (1\leq i\leq 2N).
    \end{equation}
    Then, by formula \eqref{eq:alternating-model} (cf.\ Figures~\ref{fig:gamma_b_v2} and \ref{alternating-partition}) we deduce that
    \begin{equation}\label{eq:0710-3}
        \tfrac{1}{p-1}\tilde{L}_i<\tilde{\ell}_i<\tilde{L}_i \quad (1\leq i\leq 2N),
    \end{equation}
    and also that the tangent vector $\gamma'$ and the curvature $k$ of $\gamma$ satisfy
    \begin{equation}\nonumber\label{eq:0710-4}
        \gamma'(s_{2i})=-e_1, \quad |k(s_{2i})|=2\neq0 \quad (1\leq i\leq N).
    \end{equation}
    This together with the $C^1$-convergence $\gamma_\nu\to\gamma$ implies that we can pick sequences $\{s_{2i,\nu}\}_{\nu=1}^\infty\subset[0,L]$ such that 
    \begin{equation}\label{eq:0710-5}
        \gamma_\nu'(s_{2i,\nu})=-e_1, \quad \lim_{\nu\to\infty}s_{2i,\nu}=s_{2i} \quad  (1\leq i\leq N).
    \end{equation}
    Now, for all large $\nu$ we can define a partition $\{s_{i,\nu}\}_{i=1}^{2N+1}$ of $[0,L]$ by taking $s_{1,\nu}:=0$, $s_{2N+1,\nu}:=L$, and $s_{2i-1,\nu}:=s_{2i-1}$ for $i\in\{2,\dots,N\}$.
    Let
    \begin{equation}\label{eq:0710-6}
        \tilde{L}_{i,\nu}:=s_{i+1,\nu}-s_{i,\nu}, \quad \tilde{\ell}_{i,\nu}:=\big(\gamma_\nu(s_{i+1,\nu})-\gamma_\nu(s_{i,\nu})\big)\cdot e_1 \quad (1\leq i\leq 2N).
    \end{equation}
    Then, since $\gamma_\nu\in\Ap(P_0,P_1,L)$ with $(P_1-P_0)\cdot e_1=\ell$, we have
    \begin{equation}\label{eq:0710-7}
        \sum_{i=1}^{2N}\tilde{L}_{i,\nu}=L, \quad \sum_{i=1}^{2N}\tilde{\ell}_{i,\nu}=\ell.
    \end{equation}
    In addition, by \eqref{eq:0710-2}, \eqref{eq:0710-3}, \eqref{eq:0710-5}, and the $C^0$-convergence $\gamma_\nu\to\gamma$ we deduce that for all large $\nu$,
    \begin{equation}\label{eq:0710-8}
        \tfrac{1}{p-1}\tilde{L}_{i,\nu}<\tilde{\ell}
_{i,\nu}<\tilde{L}_{i,\nu} \quad (1\leq i\leq 2N).
    \end{equation}
    Therefore, in view of \eqref{eq:0710-5}, \eqref{eq:0710-6}, \eqref{eq:0710-7}, and \eqref{eq:0710-8}, for all large $\nu$ the curve $\gamma_\nu\in\Ap$ has a $2N$ partition $\{\gamma_\nu|_{[s_{i,\nu},s_{i+1,\nu}]}\}_{i=1}^{2N}$ as in the assumption of Lemma \ref{lem:relaxation} (with $N$ replaced by $2N$), and hence
    $$\B_p[\gamma_\nu] = \sum_{i=1}^{2N}\B_p[\gamma_\nu|_{[s_{i,\nu},s_{i+1,\nu}]}] \geq \frac{C_p(2N)^p}{(L-\ell)^{p-1}}=\B_p[\gamma],$$
    where in the last part we used \eqref{eq:0710-9}.
    The proof is complete.
\end{proof}

\begin{figure}[htbp]
\centering
\includegraphics[width=80mm]{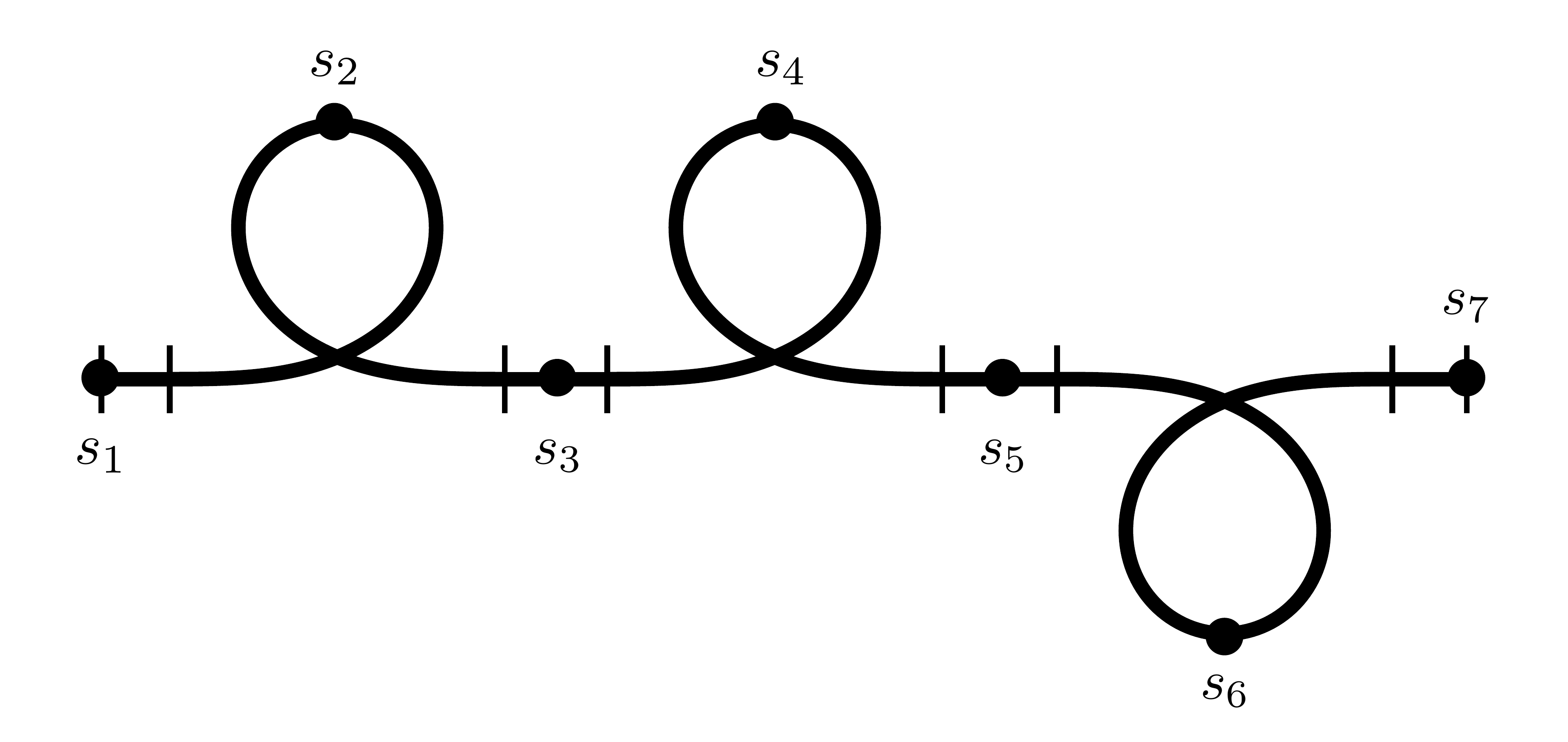} 
\caption{Decomposition of an alternating flat-core $p$-elastica.}
\label{alternating-partition}
\end{figure}

\begin{proof}[Proof of Theorem \ref{thm:main_stabilization}]
    This immediately follows by Theorem \ref{thm:alternating_detailed}. In particular, the uncountability follows by the freedom of $\boldsymbol{L}$, while the divergence of the energy follows by taking $N\to\infty$ in \eqref{eq:0710-9}.
\end{proof}

\begin{remark}[Rigidity]
    In fact our argument can also imply the following rigidity: If a curve $\gamma\in\Ap$ lies in a small neighborhood of a given alternating flat-core $p$-elastica $\bar{\gamma}$, and if $\B_p[\gamma]=\B_p[\bar{\gamma}]$, then $\gamma$ is also an alternating flat-core $p$-elastica (possibly different from $\bar{\gamma}$).
    This is mainly due to the uniqueness in Theorem \ref{thm:min_free_half}.
\end{remark}

\appendix
\section{First variation arguments for hooked $p$-elasticae}\label{sect:appendix-A}

Here we check that any (local) minimizer of $\mathcal{B}_p$ in $\Ah(\ell,L)$ is a hooked $p$-elastica. 
Let $W^{2,p}_{\rm imm}(0,1;\R^2)$ denote the set of immersed $W^{2,p}$-curves, i.e.,  
\[
W^{2,p}_{\rm imm}(0,1;\R^2):=\Set{ \gamma \in W^{2,p}(0,1;\R^2) |\, |\gamma'(t)|\neq0 \ \ \text{for all} \ \ t\in[0,1] }, 
\]
and define an immersed counterpart of $\Ah$ by
\begin{align*}
    \begin{split}
            \Ah^* &= \Ah^*(\ell,L)\\
            &:=\Set{ \gamma \in W^{2,p}_\mathrm{imm}(0,1; \R^2) | (\gamma(1)-\gamma(0))\cdot e_1 =\ell,\ \mathcal{L}[\gamma]=L,\ \mathbf{t}(L)=-e_1 },
    \end{split}
\end{align*}
where $\mathcal{L}$ denotes the length functional, i.e., $\mathcal{L}[\gamma]:=\int_\gamma\,ds$, and $\mathbf{t}:[0,L]\to\mathbf{S}^1$ denotes the unit tangent.
It is clear that if $\gamma$ is a (local) minimizer of $\mathcal{B}_p$ in $\Ah(\ell,L)$, then the curve $\bar{\gamma}$ defined by $\bar{\gamma}(t):=\gamma(Lt)$ for $t\in[0,1]$ is a (local) minimizer of $\mathcal{B}_p$ in $\Ah^*$, thus in particular a critical point of $\mathcal{B}_p$ in $\Ah^*$. 
Here we define:
\begin{itemize}
\item For $\gamma\in \Ah^*$, we call a one-parameter family $\varepsilon\mapsto \gamma_\varepsilon \in \Ah^*$ \emph{admissible perturbation} of $\gamma$ in $\Ah^*$ if $\gamma_0=\gamma$ and if the derivative $\frac{d}{d\varepsilon}\gamma_\varepsilon\big|_{\varepsilon=0}$ exists.
\item We say that $\gamma\in \Ah^*$ is a \emph{critical point} of $\mathcal{B}_p$ in  $\Ah^*$ if
    for any admissible perturbation $(\varepsilon\mapsto\gamma_\varepsilon)$ of $\gamma$ in $\Ah^*$ the first variation of $\mathcal{B}_p$ vanishes:
\begin{align}\notag
\frac{d}{d\varepsilon}\mathcal{B}_p[\gamma_\varepsilon]\Big|_{\varepsilon=0}=0.
\end{align}
\end{itemize}
Hence, in order to check that any (local) minimizer of $\mathcal{B}_p$ in $\Ah$ is a hooked $p$-elastica, it suffices to show that the arclength parametrized signed curvature $k$ of any critical point of $\mathcal{B}_p$ in $\Ah^*$ satisfies the Euler--Lagrange equation \eqref{eq:EL} and that $w:=|k|^{p-2}k$ satisfies $w(0)=w'(L)=0$. 
(Recall $w \in C^2([0,L])$ by  Proposition~\ref{prop:regularity-w}.)

First we deduce a weak form of the Euler--Lagrange equation for hooked $p$-elasticae.
Note carefully that, although the general flow of the argument below is similar to our previous studies \cite{MY_AMPA,MR4852697}, we need a slightly different approximation procedure since the boundary condition is of higher order.

\begin{lemma}[The Euler--Lagrange equation for hooked $p$-elasticae]\label{lem:EL-hooked}
Let $\gamma \in \Ah^*$ be a critical point of $\mathcal{B}_p$ in $\Ah^*$.
Then the (arclength parametrized) signed curvature $k$ of $\gamma$ satisfies $k\in L^\infty(0,L)$ and there exists $\lambda\in \R$ such that 
$k$ satisfies
\eqref{eq:EL} 
for all $\varphi \in W^{2,p}(0,L)$ with $\varphi(0)=0$ and $\varphi'(L)=0$.
\end{lemma}
\begin{proof}
By the Lagrange multiplier method (cf.\ \cite[Proposition 43.21]{Zeid3}), there is a multiplier $\lambda \in \R$ such that 
\begin{align}\label{eq:Fderivative}
\big\langle D\mathcal{B}_p[\gamma] + \lambda D\mathcal{L}[\gamma], h \big\rangle=0
\end{align}
for all $h\in W^{2,p}(0,1;\R^2)$ with $(h(1)-h(0))\cdot e_1=0$ and $h'(1)\cdot e_2=0$, where $D\mathcal{B}_p[\gamma]$ and $D\mathcal{L}[\gamma]$ are the Fr\'echet derivatives of $\mathcal{B}_p$ and $\mathcal{L}$ at $\gamma$, respectively.
By the known computation of the first derivative of $\B_p$ (cf.\ \cite{MR4852697}*{Lemmas~A.3 and A.4}) and the change of variables $\eta:=h\circ\sigma^{-1}$ with $\sigma(t):=\int_0^t|\gamma'|$, we can  rewrite \eqref{eq:Fderivative} in terms of the arclength parametrization $\tilde{\gamma}$ of $\gamma \in \Ah^*$: 
\begin{align}\label{eq:0525-1}
\int_0^L\Big( (1-2p) |\tilde{\gamma}''|^p (\tilde{\gamma}'\cdot \eta') +p|\tilde{\gamma}''|^{p-2}(\tilde{\gamma}'' \cdot \eta'') +\lambda (\tilde{\gamma}' \cdot \eta')\Big) ds = 0 
\end{align}
for all $\eta\in W^{2,p}(0,L; \R^2)$ with
\begin{align}\label{eq:testfct-BC-h}
    (\eta(L)-\eta(0)) \cdot e_1=0, \quad \eta'(L)\cdot e_2 =0.
\end{align}
Let $k:[0,L]\to\R$ be the arclength parametrized signed curvature of $\gamma$.
By \cite[Proposition 2.1]{MY_AMPA} it directly follows that 
$k\in L^\infty(0,L)$ as well as that \eqref{eq:EL} holds for $\varphi\in C^\infty_{\rm c}(0,L)$, or equivalently for $\varphi \in W^{2,p}_0(0,L)$ up to approximation.
In what follows we check that the boundary conditions for $\varphi$ can be relaxed.

Fix $\varphi \in W^{2,p}(0,L)$ with $\varphi(0)=\varphi'(L)=0$ arbitrarily.
Let $\mathbf{t}$ and $\mathbf{n}$ be the unit tangent vector and the unit normal vector of $\tilde{\gamma}$, respectively, defined by $\mathbf{t}(s):=\partial_s\tilde{\gamma}(s)$ and $\mathbf{n}(s):=R_{\pi/2}\mathbf{t}(s)$, where $R_{\theta}$ stands for the counterclockwise rotation matrix through angle $\theta\in \R$.
Recall that the Frenet--Serret formula yields $\mathbf{t}'(s)=k(s)\mathbf{n}(s)$ and $\mathbf{n}'(s)=-k(s)\mathbf{t}(s)$, and in particular
\[\mathbf{n}(s)= \mathbf{n}(0)-\int_0^s k(\sigma) \mathbf{t}(\sigma)\,d\sigma, \quad s\in[0,L].\]
Take any sequence $\{ k_j \}_{j\in \N} \subset C^1(0,L)$ such that
$k_j\to k$ in $L^p(0,L)$ as $j\to \infty$.
For each $j\in \N$, set
\begin{align*}
\mathbf{n}_j (s) &:= \mathbf{n}(0)-\int_0^s k_j(\sigma) \mathbf{t}(\sigma)\,d\sigma.
\end{align*} 
Then we see that $\{\mathbf{n}_j\}_{j\in\N} \subset C^2(0,L)\cap C([0,L])$. 
Since $k_j\to k $ in $L^p(0,L)$, we have $\mathbf{n}_j \to \mathbf{n}$ in $C([0,L])$.
Let $f\in C^2([0,L])$ be a function satisfying $f(0)=f'(0)=f'(L)=0$ and $f(L)=1$. 
For each $j\in \N$, set $r_j(s):=f(s)\varphi(L)( \mathbf{n}(L)- \mathbf{n}_j(L))$ and 
\begin{align*}
\eta_j (s) := \varphi(s)\mathbf{n}_j (s) + r_j(s).
\end{align*} 
Note that $\eta_j(0)=\varphi(0)\mathbf{n}(0)$ and $\eta_j(L)=\varphi(L)\mathbf{n}(L)$.
Using again the Frenet--Serret formula, we have
\begin{align*}
\eta_j'(s)&=\vp'(s)\mathbf{n}_j(s)-\vp(s)k_j(s)\mathbf{t}(s)+r'_j(s), \\
\eta_j''(s)&=\vp''(s)\mathbf{n}_j(s)-2\vp'(s)k_j(s)\mathbf{t}(s) -\vp(s)k_j'(s)\mathbf{t}(s) - \vp(s)k_j(s)k(s)\mathbf{n}(s)+r''_j(s), 
\end{align*}
which implies that $\{\eta_j\}_{j\in \N}\subset W^{2,p}(0,L;\R^2)$.
Since $\gamma \in \Ah^*$, we obtain $\mathbf{t}(L)=-e_1$, and this also means that $\mathbf{n}(L)=-e_2$. 
Combining this with $\varphi(0)=\varphi'(L)=0$, we see that $\eta_j$ satisfies \eqref{eq:testfct-BC-h} for each $j\in \N$.

Note that
 \begin{align*}
  (1-2p) |\tilde{\gamma}''|^p (\tilde{\gamma}' \cdot \eta_j') &=(1-2p)|k|^p \Big(  (\tilde{\gamma}' \cdot \varphi' \mathbf{n}_j) - k_j\vp +\tilde{\gamma}'\cdot r'_j \Big), \\ 
 p|\tilde{\gamma}''|^{p-2}(\tilde{\gamma}''\cdot \eta_j'') &= p |k|^{p-2}\Big( (\tilde{\gamma}''\cdot \varphi''\mathbf{n}_j)- |k|^{2} k_j\varphi + \tilde{\gamma}''\cdot r''_j \Big),
 \end{align*}
where $|\tilde{\gamma}'|\equiv|\mathbf{n}|\equiv1$, $|\tilde{\gamma}''(s)|=|k(s)|$, and $(\tilde{\gamma}'', \mathbf{t})=0$ were used.
Substituting $\eta=\eta_j$ into \eqref{eq:0525-1}, we obtain
\begin{align*}
\int_0^L \bigg(& (1-2p)|k|^p \Big( \vp'(\tilde{\gamma}'\cdot \mathbf{n}_j) - k_j\vp +\tilde{\gamma}'\cdot r'_j\Big) \\
&+p|k|^{p-2} \Big( \vp''(\tilde{\gamma}''\cdot \mathbf{n}_j) - |k|^{2} k_j\varphi +\tilde{\gamma}''\cdot r''_j\Big) 
+\lambda \Big( \vp'(\tilde{\gamma}'\cdot \mathbf{n}_j)-k_j\varphi\Big)
\bigg) ds = 0.
\end{align*}
By using $k_j\to k$ in $L^p(0,L)$, $\mathbf{n}_j\to \mathbf{n}$ in $C([0,L];\R^2)$, $r_j\to0$ in $C^2([0,L];\R^2)$ and $k \in L^{\infty}(0,L)$, we can obtain \eqref{eq:EL} as the limit of the above equality.
\end{proof}

From this we deduce the additional natural boundary condition:
\begin{lemma}[Improved regularity and natural boundary condition]\label{lem:BC-hooked-w}
Let $\gamma\in \Ah^*$ be a critical point of $\mathcal{B}_p$ in $\Ah^*$ and $k:[0,L]\to\R$ be the arclength parametrized signed curvature of $\gamma$. 
Then $k\in C([0,L])$.
In addition, the function $w:=|k|^{p-2}k$ is of class $C^2$ and satisfies
\[
w(0)=0, \quad w'(L)=0.
\]
\end{lemma}
\begin{proof}
Let $\gamma\in \Ah^*$ be a critical point of $\mathcal{B}_p$ in $\Ah^*$.
By Lemma~\ref{lem:EL-hooked}, there exists $\lambda\in \R$ such that the arclength parametrized signed curvature $k$ of $\gamma$ satisfies \eqref{eq:EL} for all $\varphi\in C^\infty_{\rm c}(0,L)$ in particular, and hence $\gamma$ is a $p$-elastica so that $k\in C([0,L])$ and $w:=|k|^{p-2}k\in C^2([0,L])$ by Proposition \ref{prop:regularity-w}.
In addition, we deduce from Lemma~\ref{lem:EL-hooked} that
\begin{align*}
\int_0^L \! \Big(p  w \varphi'' + (p-1)|w|^{\frac{2}{p-1}}w \varphi - \lambda |w|^{\frac{2-p}{p-1}}w \varphi \Big) ds = 0
\end{align*}
for $\varphi \in W^{2,p}(0,L)$ with $\varphi(0)=\varphi'(L)=0$.
Using the boundary condition for $\varphi$ together with the regularity $w\in C^2([0,L])$, we further deduce via integration by parts that
\begin{align*}
0 
&= \big[ p w(s) \varphi'(s) - p w'(s) \varphi(s) \big]_{s=0}^{s=L} +\int_0^L\! \Big(pw'' + (p-1)|w|^{\frac{2}{p-1}}w - \lambda |w|^{\frac{2-p}{p-1}}w  \Big)\vp \,ds \\
&= -pw'(L)\varphi(L)-pw(0)\varphi'(0).
\end{align*}
Then, with the choice of $\varphi$ satisfying $\varphi(L)=1$ and $\varphi'(0)=0$ (resp.\ $\varphi(L)=0$ and $\varphi'(0)=1$), we obtain $w'(L)=0$ (resp.\ $w(0)=0$).
\end{proof}


\bibliography{ref_Miura-Yoshizawa-ver6}

\begin{bibdiv}
\begin{biblist}

\bib{AM17}{article}{
      author={Acerbi, Emilio},
      author={Mucci, Domenico},
       title={Curvature-dependent energies: a geometric and analytical
  approach},
        date={2017},
        ISSN={0308-2105},
     journal={Proc. Roy. Soc. Edinburgh Sect. A},
      volume={147},
      number={3},
       pages={449\ndash 503},
         url={https://doi.org/10.1017/S0308210516000202},
      review={\MR{3656702}},
}

\bib{AM03}{article}{
      author={Ambrosio, Luigi},
      author={Masnou, Simon},
       title={A direct variational approach to a problem arising in image
  reconstruction},
        date={2003},
        ISSN={1463-9963},
     journal={Interfaces Free Bound.},
      volume={5},
      number={1},
       pages={63\ndash 81},
         url={https://doi.org/10.4171/IFB/72},
      review={\MR{1959769}},
}

\bib{AGM}{article}{
      author={Arroyo, J.},
      author={Garay, O.~J.},
      author={Menc\'{\i}a, J.~J.},
       title={Closed generalized elastic curves in {$S^2(1)$}},
        date={2003},
        ISSN={0393-0440},
     journal={J. Geom. Phys.},
      volume={48},
      number={2-3},
       pages={339\ndash 353},
         url={https://doi.org/10.1016/S0393-0440(03)00047-0},
      review={\MR{2007599}},
}

\bib{Bar52}{article}{
      author={Barenblatt, G.~I.},
       title={On self-similar motions of a compressible fluid in a porous
  medium},
        date={1952},
     journal={Akad. Nauk SSSR. Prikl. Mat. Meh.},
      volume={16},
       pages={679\ndash 698},
      review={\MR{0052948}},
}

\bib{BDP93}{article}{
      author={Bellettini, G.},
      author={Dal~Maso, G.},
      author={Paolini, M.},
       title={Semicontinuity and relaxation properties of a curvature depending
  functional in {$2$}{D}},
        date={1993},
        ISSN={0391-173X},
     journal={Ann. Scuola Norm. Sup. Pisa Cl. Sci. (4)},
      volume={20},
      number={2},
       pages={247\ndash 297},
         url={http://www.numdam.org/item?id=ASNSP_1993_4_20_2_247_0},
      review={\MR{1233638}},
}

\bib{BM04}{article}{
      author={Bellettini, G.},
      author={Mugnai, L.},
       title={Characterization and representation of the lower semicontinuous
  envelope of the elastica functional},
        date={2004},
        ISSN={0294-1449},
     journal={Ann. Inst. H. Poincar\'{e} C Anal. Non Lin\'{e}aire},
      volume={21},
      number={6},
       pages={839\ndash 880},
         url={https://doi.org/10.1016/j.anihpc.2004.01.001},
      review={\MR{2097034}},
}

\bib{BM07}{article}{
      author={Bellettini, G.},
      author={Mugnai, L.},
       title={A varifolds representation of the relaxed elastica functional},
        date={2007},
        ISSN={0944-6532},
     journal={J. Convex Anal.},
      volume={14},
      number={3},
       pages={543\ndash 564},
      review={\MR{2341303}},
}

\bib{BHV}{article}{
      author={Blatt, Simon},
      author={Hopper, Christopher},
      author={Vorderobermeier, Nicole},
       title={A regularized gradient flow for the {$p$}-elastic energy},
        date={2022},
        ISSN={2191-9496},
     journal={Adv. Nonlinear Anal.},
      volume={11},
      number={1},
       pages={1383\ndash 1411},
         url={https://doi.org/10.1515/anona-2022-0244},
      review={\MR{4418802}},
}

\bib{BVH}{article}{
      author={Blatt, Simon},
      author={Hopper, Christopher~P.},
      author={Vorderobermeier, Nicole},
       title={A minimising movement scheme for the {$p$}-elastic energy of
  curves},
        date={2022},
        ISSN={1424-3199},
     journal={J. Evol. Equ.},
      volume={22},
      number={2},
       pages={Paper No. 41, 25},
         url={https://doi.org/10.1007/s00028-022-00791-w},
      review={\MR{4416790}},
}

\bib{CV21}{incollection}{
      author={Ciani, Simone},
      author={Vespri, Vincenzo},
       title={An introduction to {B}arenblatt solutions for anisotropic
  {$p$}-{L}aplace equations},
        date={2021},
   booktitle={Anomalies in partial differential equations},
      series={Springer INdAM Ser.},
      volume={43},
   publisher={Springer, Cham},
       pages={99\ndash 125},
         url={https://doi.org/10.1007/978-3-030-61346-4_5},
      review={\MR{4219198}},
}

\bib{DFLM}{article}{
      author={Dal~Maso, G.},
      author={Fonseca, I.},
      author={Leoni, G.},
      author={Morini, M.},
       title={A higher order model for image restoration: the one-dimensional
  case},
        date={2009},
        ISSN={0036-1410},
     journal={SIAM J. Math. Anal.},
      volume={40},
      number={6},
       pages={2351\ndash 2391},
         url={https://doi.org/10.1137/070697823},
      review={\MR{2481298}},
}

\bib{DMOY}{article}{
      author={Dall'Acqua, Anna},
      author={M\"uller, Marius},
      author={Okabe, Shinya},
      author={Yoshizawa, Kensuke},
       title={An obstacle problem for the {$p$}-elastic energy},
        date={2024},
        ISSN={0944-2669,1432-0835},
     journal={Calc. Var. Partial Differential Equations},
      volume={63},
      number={6},
       pages={Paper No. 145, 43},
         url={https://doi.org/10.1007/s00526-024-02752-2},
      review={\MR{4761861}},
}

\bib{DiB93_book}{book}{
      author={DiBenedetto, Emmanuele},
       title={Degenerate parabolic equations},
      series={Universitext},
   publisher={Springer-Verlag, New York},
        date={1993},
        ISBN={0-387-94020-0},
         url={https://doi.org/10.1007/978-1-4612-0895-2},
      review={\MR{1230384}},
}

\bib{DSSVV05}{article}{
      author={Du, Quan},
      author={Smith, Chaim},
      author={Shiffeldrim, Nahum},
      author={Vologodskaia, Maria},
      author={Vologodskii, Alexander},
       title={Cyclization of short {DNA} fragments and bending fluctuations of
  the double helix},
        date={2005},
     journal={Proc. Natl. Acad. Sci. U.S.A},
      volume={102},
       pages={5397\ndash 5402},
}

\bib{FKN18}{article}{
      author={Ferone, Vincenzo},
      author={Kawohl, Bernd},
      author={Nitsch, Carlo},
       title={Generalized elastica problems under area constraint},
        date={2018},
        ISSN={1073-2780},
     journal={Math. Res. Lett.},
      volume={25},
      number={2},
       pages={521\ndash 533},
         url={https://doi.org/10.4310/MRL.2018.v25.n2.a9},
      review={\MR{3826833}},
}

\bib{FigalliZhang22}{article}{
      author={Figalli, Alessio},
      author={Zhang, Yi Ru-Ya},
       title={Sharp gradient stability for the {S}obolev inequality},
        date={2022},
        ISSN={0012-7094,1547-7398},
     journal={Duke Math. J.},
      volume={171},
      number={12},
       pages={2407\ndash 2459},
         url={https://doi.org/10.1215/00127094-2022-0051},
      review={\MR{4484209}},
}

\bib{GPT23}{article}{
      author={Gruber, Anthony},
      author={P\'{a}mpano, \'{A}lvaro},
      author={Toda, Magdalena},
       title={Instability of closed {$p$}-elastic curves in {$\Bbb{S}^2$}},
        date={2023},
        ISSN={0219-5305,1793-6861},
     journal={Anal. Appl. (Singap.)},
      volume={21},
      number={6},
       pages={1533\ndash 1559},
         url={https://doi.org/10.1142/S0219530523500173},
      review={\MR{4663707}},
}

\bib{GV88}{article}{
      author={Guedda, Mohammed},
      author={V\'{e}ron, Laurent},
       title={Bifurcation phenomena associated to the {$p$}-{L}aplace
  operator},
        date={1988},
        ISSN={0002-9947},
     journal={Trans. Amer. Math. Soc.},
      volume={310},
      number={1},
       pages={419\ndash 431},
         url={https://doi.org/10.2307/2001132},
      review={\MR{965762}},
}

\bib{HHH20}{article}{
      author={He, Tieshan},
      author={He, Lang},
      author={Huang, Yehui},
       title={Infinitely many nodal solutions for generalized logistic
  equations without odd symmetry on reaction},
        date={2020},
        ISSN={0362-546X,1873-5215},
     journal={Nonlinear Anal.},
      volume={195},
       pages={111741, 20},
         url={https://doi.org/10.1016/j.na.2019.111741},
      review={\MR{4053760}},
}

\bib{HKM06_book}{book}{
      author={Heinonen, Juha},
      author={Kilpel\"{a}inen, Tero},
      author={Martio, Olli},
       title={Nonlinear potential theory of degenerate elliptic equations},
   publisher={Dover Publications, Inc., Mineola, NY},
        date={2006},
        ISBN={0-486-45050-3},
      review={\MR{2305115}},
}

\bib{KV88}{article}{
      author={Kamin, S.},
      author={V\'{a}zquez, J.~L.},
       title={Fundamental solutions and asymptotic behaviour for the
  {$p$}-{L}aplacian equation},
        date={1988},
        ISSN={0213-2230},
     journal={Rev. Mat. Iberoamericana},
      volume={4},
      number={2},
       pages={339\ndash 354},
         url={https://doi.org/10.4171/RMI/77},
      review={\MR{1028745}},
}

\bib{KV93}{article}{
      author={Kamin, Shoshana},
      author={V\'{e}ron, Laurent},
       title={Flat core properties associated to the {$p$}-{L}aplace operator},
        date={1993},
        ISSN={0002-9939},
     journal={Proc. Amer. Math. Soc.},
      volume={118},
      number={4},
       pages={1079\ndash 1085},
         url={https://doi.org/10.2307/2160060},
      review={\MR{1139470}},
}

\bib{Lev}{article}{
      author={Levien, Raph},
       title={The elastica: a mathematical history},
        date={2008},
     journal={Technical Report No. UCB/EECS-2008-10, University of California,
  Berkeley},
}

\bib{Lind19}{book}{
      author={Lindqvist, Peter},
       title={Notes on the stationary {$p$}-{L}aplace equation},
      series={SpringerBriefs in Mathematics},
   publisher={Springer, Cham},
        date={2019},
        ISBN={978-3-030-14500-2; 978-3-030-14501-9},
         url={https://doi.org/10.1007/978-3-030-14501-9},
      review={\MR{3931688}},
}

\bib{LP22}{article}{
      author={L\'{o}pez, Rafael},
      author={P\'{a}mpano, \'{A}lvaro},
       title={Stationary soap films with vertical potentials},
        date={2022},
        ISSN={0362-546X},
     journal={Nonlinear Anal.},
      volume={215},
       pages={Paper No. 112661, 22},
         url={https://doi.org/10.1016/j.na.2021.112661},
      review={\MR{4340791}},
}

\bib{Maddocks1981}{thesis}{
      author={Maddocks, J.~H.},
       title={Analysis of nonlinear differential equations governing the
  equilibria of an elastic rod and their stability},
        type={Ph.D. Thesis},
        date={1981},
}

\bib{Maddocks1984}{article}{
      author={Maddocks, J.~H.},
       title={Stability of nonlinearly elastic rods},
        date={1984},
        ISSN={0003-9527},
     journal={Arch. Rational Mech. Anal.},
      volume={85},
      number={4},
       pages={311\ndash 354},
         url={https://doi.org/10.1007/BF00275737},
      review={\MR{746170}},
}

\bib{MM98}{inproceedings}{
      author={Masnou, S.},
      author={Morel, J.-M.},
       title={Level lines based disocclusion},
        date={1998},
   booktitle={Proceedings 1998 international conference on image processing},
      volume={3},
       pages={259\ndash 263},
}

\bib{MN13}{article}{
      author={Masnou, S.},
      author={Nardi, G.},
       title={A coarea-type formula for the relaxation of a generalized
  elastica functional},
        date={2013},
        ISSN={0944-6532},
     journal={J. Convex Anal.},
      volume={20},
      number={3},
       pages={617\ndash 653},
      review={\MR{3136594}},
}

\bib{Miura20}{article}{
      author={Miura, Tatsuya},
       title={Elastic curves and phase transitions},
        date={2020},
        ISSN={0025-5831},
     journal={Math. Ann.},
      volume={376},
      number={3-4},
       pages={1629\ndash 1674},
         url={https://doi.org/10.1007/s00208-019-01821-8},
      review={\MR{4081125}},
}

\bib{Miura_LiYau}{article}{
      author={Miura, Tatsuya},
       title={Li-{Y}au type inequality for curves in any codimension},
        date={2023},
        ISSN={0944-2669,1432-0835},
     journal={Calc. Var. Partial Differential Equations},
      volume={62},
      number={8},
       pages={Paper No. 216, 28},
         url={https://doi.org/10.1007/s00526-023-02559-7},
      review={\MR{4631455}},
}

\bib{MO21}{article}{
      author={Miura, Tatsuya},
      author={Okabe, Shinya},
       title={On the isoperimetric inequality and surface diffusion flow for
  multiply winding curves},
        date={2021},
        ISSN={0003-9527},
     journal={Arch. Ration. Mech. Anal.},
      volume={239},
      number={2},
       pages={1111\ndash 1129},
         url={https://doi.org/10.1007/s00205-020-01591-7},
      review={\MR{4201623}},
}

\bib{MY_AMPA}{article}{
      author={Miura, Tatsuya},
      author={Yoshizawa, Kensuke},
       title={Complete classification of planar {$p$}-elasticae},
        date={2024},
        ISSN={0373-3114,1618-1891},
     journal={Ann. Mat. Pura Appl. (4)},
      volume={203},
      number={5},
       pages={2319\ndash 2356},
         url={https://doi.org/10.1007/s10231-024-01445-z},
      review={\MR{4797294}},
}

\bib{MY_Crelle}{article}{
      author={Miura, Tatsuya},
      author={Yoshizawa, Kensuke},
       title={General rigidity principles for stable and minimal elastic
  curves},
        date={2024},
        ISSN={0075-4102,1435-5345},
     journal={J. Reine Angew. Math.},
      volume={810},
       pages={253\ndash 281},
         url={https://doi.org/10.1515/crelle-2024-0018},
      review={\MR{4739246}},
}

\bib{MR4852697}{article}{
      author={Miura, Tatsuya},
      author={Yoshizawa, Kensuke},
       title={Pinned planar {$p$}-elasticae},
        date={2024},
        ISSN={0022-2518,1943-5258},
     journal={Indiana Univ. Math. J.},
      volume={73},
      number={6},
       pages={2155\ndash 2208},
      review={\MR{4852697}},
}

\bib{NP20}{article}{
      author={Novaga, Matteo},
      author={Pozzi, Paola},
       title={A second order gradient flow of {$p$}-elastic planar networks},
        date={2020},
        ISSN={0036-1410},
     journal={SIAM J. Math. Anal.},
      volume={52},
      number={1},
       pages={682\ndash 708},
         url={https://doi.org/10.1137/19M1262292},
      review={\MR{4062804}},
}

\bib{OPW20}{article}{
      author={Okabe, Shinya},
      author={Pozzi, Paola},
      author={Wheeler, Glen},
       title={A gradient flow for the {$p$}-elastic energy defined on closed
  planar curves},
        date={2020},
        ISSN={0025-5831},
     journal={Math. Ann.},
      volume={378},
      number={1-2},
       pages={777\ndash 828},
         url={https://doi.org/10.1007/s00208-019-01885-6},
      review={\MR{4150936}},
}

\bib{OW23}{article}{
      author={Okabe, Shinya},
      author={Wheeler, Glen},
       title={The {{\(p\)}}-elastic flow for planar closed curves with constant
  parametrization},
    language={English},
        date={2023},
        ISSN={0021-7824},
     journal={J. Math. Pures Appl. (9)},
      volume={173},
       pages={1\ndash 42},
}

\bib{Poz20}{article}{
      author={Pozzetta, Marco},
       title={A varifold perspective on the {$p$}-elastic energy of planar
  sets},
        date={2020},
        ISSN={0944-6532},
     journal={J. Convex Anal.},
      volume={27},
      number={3},
       pages={845\ndash 879},
      review={\MR{4158524}},
}

\bib{Poz22}{article}{
      author={Pozzetta, Marco},
       title={Convergence of elastic flows of curves into manifolds},
        date={2022},
        ISSN={0362-546X},
     journal={Nonlinear Anal.},
      volume={214},
       pages={Paper No. 112581, 53},
  url={https://doi-org.ezproxy.lib.kyushu-u.ac.jp/10.1016/j.na.2021.112581},
      review={\MR{4318851}},
}

\bib{Rynne19}{article}{
      author={Rynne, Bryan~P.},
       title={Linearized stability implies asymptotic stability for radially
  symmetric equilibria of {$p$}-{L}aplacian boundary value problems in the unit
  ball in {$\Bbb R^N$}},
        date={2019},
        ISSN={1072-6691},
     journal={Electron. J. Differential Equations},
       pages={Paper No. 94, 17},
      review={\MR{3992305}},
}

\bib{Rynne20}{article}{
      author={Rynne, Bryan~P.},
       title={Linearized stability implies dynamic stability for equilibria of
  1-dimensional, {$p$}-{L}aplacian boundary value problems},
        date={2020},
        ISSN={0308-2105,1473-7124},
     journal={Proc. Roy. Soc. Edinburgh Sect. A},
      volume={150},
      number={3},
       pages={1313\ndash 1338},
         url={https://doi.org/10.1017/prm.2018.123},
      review={\MR{4091062}},
}

\bib{Sac08_3}{article}{
      author={Sachkov, Yu.~L.},
       title={Conjugate points in the {E}uler elastic problem},
        date={2008},
        ISSN={1079-2724},
     journal={J. Dyn. Control Syst.},
      volume={14},
      number={3},
       pages={409\ndash 439},
         url={https://doi.org/10.1007/s10883-008-9044-x},
      review={\MR{2425306}},
}

\bib{Sac08_2}{article}{
      author={Sachkov, Yu.~L.},
       title={Maxwell strata in the {E}uler elastic problem},
        date={2008},
        ISSN={1079-2724},
     journal={J. Dyn. Control Syst.},
      volume={14},
      number={2},
       pages={169\ndash 234},
         url={https://doi.org/10.1007/s10883-008-9039-7},
      review={\MR{2390213}},
}

\bib{SW20}{article}{
      author={Shioji, Naoki},
      author={Watanabe, Kohtaro},
       title={Total {$p$}-powered curvature of closed curves and flat-core
  closed {$p$}-curves in {${\bf S}^2(G)$}},
        date={2020},
        ISSN={1019-8385},
     journal={Comm. Anal. Geom.},
      volume={28},
      number={6},
       pages={1451\ndash 1487},
         url={https://doi.org/10.4310/CAG.2020.v28.n6.a6},
      review={\MR{4184824}},
}

\bib{Sin}{incollection}{
      author={Singer, David~A.},
       title={Lectures on elastic curves and rods},
        date={2008},
   booktitle={Curvature and variational modeling in physics and biophysics},
      series={AIP Conf. Proc.},
      volume={1002},
   publisher={Amer. Inst. Phys., Melville, NY},
       pages={3\ndash 32},
         url={https://doi.org/10.1063/1.2918095},
      review={\MR{2483890}},
}

\bib{Tak00}{article}{
      author={Takeuchi, Shingo},
       title={Behavior of solutions near the flat hats of stationary solutions
  for a degenerate parabolic equation},
        date={2000},
        ISSN={0036-1410,1095-7154},
     journal={SIAM J. Math. Anal.},
      volume={31},
      number={3},
       pages={678\ndash 692},
         url={https://doi.org/10.1137/S003614109834257X},
      review={\MR{1745482}},
}

\bib{Takeuchi12}{article}{
      author={Takeuchi, Shingo},
       title={Generalized {J}acobian elliptic functions and their application
  to bifurcation problems associated with {$p$}-{L}aplacian},
        date={2012},
        ISSN={0022-247X},
     journal={J. Math. Anal. Appl.},
      volume={385},
      number={1},
       pages={24\ndash 35},
         url={https://doi.org/10.1016/j.jmaa.2011.06.063},
      review={\MR{2832071}},
}

\bib{Takeuchi_RIMS}{article}{
      author={Takeuchi, Shingo},
       title={Generalized {J}acobian elliptic functions},
        date={2013},
     journal={RIMS k\^oky\^uroku},
      volume={1838},
       pages={71\ndash 101},
  url={https://www.kurims.kyoto-u.ac.jp/~kyodo/kokyuroku/contents/pdf/1838-07.pdf},
}

\bib{TY00}{article}{
      author={Takeuchi, Shingo},
      author={Yamada, Yoshio},
       title={Asymptotic properties of a reaction-diffusion equation with
  degenerate {$p$}-{L}aplacian},
        date={2000},
        ISSN={0362-546X,1873-5215},
     journal={Nonlinear Anal.},
      volume={42},
      number={1},
       pages={41\ndash 61},
         url={https://doi.org/10.1016/S0362-546X(98)00329-0},
      review={\MR{1769251}},
}

\bib{Tru83}{article}{
      author={Truesdell, C.},
       title={The influence of elasticity on analysis: the classic heritage},
        date={1983},
        ISSN={0273-0979},
     journal={Bull. Amer. Math. Soc. (N.S.)},
      volume={9},
      number={3},
       pages={293\ndash 310},
         url={https://doi.org/10.1090/S0273-0979-1983-15187-X},
      review={\MR{714991}},
}

\bib{Vaz06_book}{book}{
      author={V\'{a}zquez, Juan~Luis},
       title={Smoothing and decay estimates for nonlinear diffusion equations},
      series={Oxford Lecture Series in Mathematics and its Applications},
   publisher={Oxford University Press, Oxford},
        date={2006},
      volume={33},
        ISBN={978-0-19-920297-3; 0-19-920297-4},
         url={https://doi.org/10.1093/acprof:oso/9780199202973.001.0001},
      review={\MR{2282669}},
}

\bib{nabe14}{article}{
      author={Watanabe, Kohtaro},
       title={Planar {$p$}-elastic curves and related generalized complete
  elliptic integrals},
        date={2014},
        ISSN={0386-5991},
     journal={Kodai Math. J.},
      volume={37},
      number={2},
       pages={453\ndash 474},
         url={https://doi.org/10.2996/kmj/1404393898},
      review={\MR{3229087}},
}

\bib{Zeid3}{book}{
      author={Zeidler, Eberhard},
       title={Nonlinear functional analysis and its applications. {III}},
   publisher={Springer-Verlag, New York},
        date={1985},
        ISBN={0-387-90915-X},
         url={https://doi.org/10.1007/978-1-4612-5020-3},
      review={\MR{768749}},
}

\end{biblist}
\end{bibdiv}

\end{document}